\documentclass{amsproc}
\usepackage{amsfonts, amsthm, amssymb, amsmath, stmaryrd}
\usepackage{mathrsfs,array}
\usepackage[utf8]{inputenc}
\usepackage{calligra}
\usepackage{eucal,color,times,enumerate,accents}
\usepackage[all]{xy}
\usepackage{mathtools}
\usepackage{url}
\usepackage{enumitem}
\usepackage{enumerate}
\usepackage[pdftex,bookmarksnumbered,bookmarksopen]
{hyperref}
\usepackage{lipsum}
\usepackage{adjustbox}
\usepackage{cite}
\usepackage[colorinlistoftodos]{todonotes}
\usepackage{tikz-cd}
\usepackage{faktor}
\usetikzlibrary{decorations.pathreplacing}
\usetikzlibrary{positioning}
\usepackage{dsfont}
\usepackage{tikz}
\usetikzlibrary{matrix,arrows,decorations.pathmorphing,positioning}
\usepackage{calligra}
\usepackage[colorinlistoftodos]{todonotes}
\usepackage{xy}
\xyoption{all}
\usepackage{tikz}
\usetikzlibrary{matrix,arrows,decorations.pathmorphing}
\usepackage{leftidx}

\DeclareMathAlphabet{\mathcalligra}{T1}{calligra}{m}{n}
\DeclareFontShape{T1}{calligra}{m}{n}{<->s*[2.2]callig15}{}

\input xy
\xyoption{all}

\usepackage{tikz-cd}
\theoremstyle{plain}
\newtheorem{thm}{Theorem}[section]
\newtheorem{conj}[thm]{Conjecture}

\newtheorem{lemma}[thm]{Lemma}
\newtheorem*{thm*}{Theorem}
\newtheorem*{question*}{Question}
\newtheorem*{conj*}{Conjecture}
\newtheorem*{cor*}{Corollary}
\newtheorem*{prop*}{Proposition}
\newtheorem{prop}[thm]{Proposition}

\newtheorem{cor}[thm]{Corollary}
\theoremstyle{definition}
\newtheorem{defi}[thm]{Definition}

\theoremstyle{remark}
\newtheorem{rmk}[thm]{Remark}

\numberwithin{equation}{section}

\newcommand{\overbar}[1]{\mkern 1.5mu\overline{\mkern-1.5mu#1\mkern-1.5mu}\mkern 1.5mu}

\DeclareMathOperator{\Br}{Br}

\DeclareMathOperator{\End}{End}

\DeclareMathOperator{\Pic}{Pic}
\DeclareMathOperator{\Hom}{Hom}
\DeclareMathOperator{\Gal}{Gal}

\DeclareMathOperator{\Spec}{Spec}

\DeclareMathOperator{\im}{Im}

\DeclareMathOperator{\Gl}{GL}

\newcommand{\MT}{\operatorname{MT}}
\newcommand{\Hdg}{\operatorname{Hdg}}
\newcommand{\Aut}{\operatorname{Aut}}
\newcommand{\NS}{\operatorname{NS}}

\newcommand{\N}{\mathbb{N}}

\newcommand{\Z}{\mathbf{Z}}
\newcommand{\Q}{\mathbf{Q}}

\newcommand{\A}{\mathbf{A}}

\newcommand{\Oo}{\mathcal{O}}

\newcommand{\et}{\textrm{\'{e}t}}

\newcommand{\C}{\mathbf{C}}

\newcommand{\adjunction}[4]{\xymatrix@1{#1{\ } \ar@<-0.3ex>[r]_{ {\scriptstyle #2}} & {\ } #3 \ar@<-0.3ex>[l]_{ {\scriptstyle #4}}}}
 \usepackage[foot]{amsaddr}
\title{Neighbors and arithmetic of isogenous K3 surfaces}
\author{Domenico Valloni}

\begin{document}
\maketitle

\begin{abstract}
We use lattice theory to study the isogeny class of a K3 surface. Starting from isotropic Brauer classes, we construct isogenies via Kneser method of neighboring lattices. We also determine the fields of definition of isogenous K3 surfaces, and study Kneser construction over number fields. We then apply our results to relate conjectures about the finiteness of Brauer groups and N\'{e}ron-Severi lattices of K3 surfaces.
\end{abstract}

\section{Introduction} 
Two complex K3 surfaces $X$ and $Y$ are said to be \textit{isogenous} if there exists a Hodge isometry $f \colon \mathrm{T}(Y,\Q) \xrightarrow{\sim} \mathrm{T}(X,\Q)$ between their rational transcendental lattices. Isogenies between K3 surfaces have been studied since Mukai's seminal work \cite{MR893604} on the moduli spaces of stable sheaves on K3 surfaces, and his ideas eventually led to a proof that every rational Hodge isometry is algebraic (see also \cite{NikulinHodgeK3} and \cite{Buskin} for the cases not considered by Mukai, and \cite{MR4014777} where isogenies are studied via twisted derived equivalences). 
Our first aim in this paper, is to apply a lattice-theoretical method due to Kneser to construct isogenies. In fact, it is difficult in general to produce K3 surfaces that are isogenous to a given $X$: Mukai's construction of isogenous K3 surfaces as moduli spaces can only yield finitely many new K3 surfaces, corresponding to transcendental lattices that contain $\mathrm{T}(X)$ with finite index and cyclic cokernel. Given an even lattice $L$, Kneser method of neighboring lattices allows one to construct new lattices $L' \subset L \otimes \Q$ that belong to the same genus of $L$ starting from isotropic lines in $L \otimes \Z/d\Z$, which we interpret as cyclic subgroups of the relevant Brauer group. To explain his construction, consider $\mathrm{T} \subset \mathrm{T}(X)$ a sublattice such that $\mathrm{T}(X)/\mathrm{T} \cong \Z / d \Z$ and such that $d$ and $\mathrm{discr}(\mathrm{T}(X))$ coprime to each other. The sublattice $\mathrm{T}$ is said to be \textit{split} if its discriminant form $A_\mathrm{T} \coloneqq \mathrm{T}^\vee/ \mathrm{T}$ splits as $A_\mathrm{T} = A_X \oplus \Z/d\Z \oplus \Z/d \Z$, where $A_X \cong \mathrm{T}(X)^\vee/ \mathrm{T}(X)$ and $\mathrm{T}^\vee$ denotes the dual of $\mathrm{T}$. The results of Kneser can be formulated as a natural one-to-one correspondence between
\begin{enumerate}
    \item Isotropic lines in $\mathrm{T}(X) \otimes \Z / d \Z \cong \Br(X)[d]$;
    \item Split sublattices of index $d$ in $\mathrm{T}(X)$;
    \item And $d$-neighbors of $\mathrm{T}(X)$, that is, even lattices $\mathrm{T} \subset \mathrm{T}(X,\Q)$ such that $$\mathrm{T}(X) / \mathrm{T}(X) \cap \mathrm{T} \cong \mathrm{T} / \mathrm{T}(X) \cap \mathrm{T} \cong \Z/d\Z.$$
\end{enumerate}
An isotropic line is by definition a cyclic subgroup $\Z/ d \Z \subset \mathrm{T}(X) \otimes \Z / d \Z$ which is isotropic with respect to the natural perfect pairing on $\mathrm{T}(X) \otimes \Z / d \Z$. It is a non-trivial result of Kneser that $d$-neighbors always belong to the same genus (see Theorem \ref{Kneser thm}). 

 In \ref{more neighbors} we generalize this correspondence to the case when $d$ and $\mathrm{discr}(\NS(X))$ are not coprime. Thanks to the global Torelli theorem for K3 surfaces \cite{MR0284440}, to any complex K3 surface $X$ and any isotropic line $\ell \subset \Br(X)[d]$ we associate a well-defined K3 surface $X_\ell$. This comes with an isogeny $f_\ell \colon \mathrm{T}(X_\ell, \Q) \rightarrow \mathrm{T}(X, \Q)$ such that the image of $\mathrm{T}(X_\ell)$ is the corresponding neighbor of $\mathrm{T}(X)$, and such that $\NS(X_\ell) \cong \NS(X)$. We call this K3 surface the $d$-neighbor of $X$ associated to $\ell$. When $\End_{\Hdg}(\mathrm{T}(X)_\Q) = \Q$ or a totally real field, then $X_\ell \cong X_{\ell'}$ if and only if $\ell = \ell'$. When $\rho(X) \geq 3$, we also construct another K3 surface $Y_\ell$, this time well-defined up to derived equivalence, endowed with a Hodge isometry $g_\ell \colon \mathrm{T}(Y_\ell, \Q) \rightarrow \mathrm{T}(X,\Q)$ such that $g_\ell(\mathrm{T}(Y_\ell)) = f_\ell(\mathrm{T}(X_\ell)) \cap \mathrm{T}(X)$ is the associated $d$-split sublattice. In particular $\mathrm{discr}(\NS(Y_\ell)) = d^2 \cdot \mathrm{discr}(\NS(X))$. When the K3 surface $Y_\ell$ exists, both the original $X$ and its neighbor $X_\ell$ can be interpreted as moduli spaces of sheaves on some Fourier-Mukai partner of $Y_\ell$ (\ref{prop Y}). 
 
 Although we do not have a precise geometric interpretation of this construction, we believe that `taking neighbors' is a very natural operation for K3 surfaces, which in some sense is the analogue of taking quotients of abelian varieties. For instance, we show that for a Kummer surface $X = \mathrm{Km}(E_1 \times E_2)$, where $E_1, E_2$ are two non-isogenous, non-CM elliptic curves, and for any odd $d$, the $d$-neighbors of $X$ are in one-to-one correspondence to $\mathrm{Km}(E_1/C_1 \times E_2/C_2)$, where $C_1 \cong C_2 \cong \Z/d\Z$.  
 
Next, we study the surfaces $X_\ell$ and $Y_\ell$ over number fields. In fact, we prove a more general criterion that allows us to understand whether an isogenous K3 surface is defined over a given number field:
\begin{thm} \label{main thm introduction}
Let $K \subset \C$ be a number field and let $X/K$ be a K3 surface. Denote by $\rho_X$ the Galois representation associated to $\mathrm{T}(\overline{X}, \A_f) = \mathrm{T}(\overline{X}, \Q) \otimes \widehat{\Z}$. Let $Y$ be a complex K3 surface and let $f \colon \mathrm{T}(Y,\Q) \rightarrow \mathrm{T}(X_\C,\Q)$ be a Hodge isometry. Then, the K3 surface $Y$, the Hodge isometry $f$, and $\NS(Y)$ can defined over $K$ if and only if:
\begin{enumerate}
    \item The $\widehat{\Z}$-lattice $f(\mathrm{T}(Y)) \otimes \widehat{\Z}$ is stable under $\rho_X$;
    \item The induced $G_K$ action on $A_Y \coloneqq \mathrm{T}(Y)^\vee / \mathrm{T}(Y)$ is trivial.
\end{enumerate}
\end{thm}
By (2) we mean the following: since $f(\mathrm{T}(Y)) \otimes \widehat{\Z}$ is stable under $G_K = \Gal(\overline{K}/K)$ by point (1), there is an induced Galois representation on $$\frac{(f(\mathrm{T}(Y)) \otimes \widehat{\Z})^{\vee}}{f(\mathrm{T}(Y)) \otimes \widehat{\Z}} \cong A_Y,$$ and we require this to be trivial. We remark that $\NS(Y)$ and $f$ can be defined after a field extension of degree bounded independently from the K3 surfaces involved ( \ref{lemma Neron-Severi} and \ref{prop def hodge isom}). This means that the difference between a smaller field of definition of $Y$ and a field where $Y$, $f$, and $\NS(Y)$ are simultaneously defined can be universally bounded (\ref{corollary minimal field}). We also remark that point (2) is the only difference between the analogue result for abelian varieties, and it is also the most difficult thing to verify in practice. Regarding neighboring K3 surfaces, we have the following corollary:
\begin{cor}
Let $K$ be a number field and let $X/K$ be a K3 surface such that $\NS(\overline{X})^{G_K} = \NS(\overline{X})$. Let $\ell \subset \Br(\overline{X})[d]$ be an isotropic line. 
\begin{enumerate}
    \item If $\ell$ is $G_K$-stable, then $(X_\ell,f, \NS(X_\ell))$ can be defined over $K$;
    \item It $G_K$ acts trivially on $\ell$ and $\rho_{X} \geq 3$, then $(Y_\ell,g_\ell, \NS(Y_\ell))$ can be defined over $K$. 
\end{enumerate}
\end{cor}
We also apply Theorem \ref{main thm introduction} to understand wheter a K3 surface $X/K$ that is geometrically a moduli space of sheaves on some other K3 surface, is also a moduli space over $K$, see Corollary \ref{corollary moduli space} and the discussion that precedes it. 

In the last part of the paper, we give some applications of our results. 
\begin{conj}[Shafarevich]  \label{Sha conj}
Only finitely many lattices can appear as $\NS(\overline{X})$ where $X$ is any K3 surface over any number field of given degree. 
\end{conj}
Coleman conjectured an analogue of \ref{Sha conj} for abelian varieties, where the N\'{e}ron-Severi lattice is replaced by the endomorphism ring of the variety. It is known that Colemann conjecture implies a uniform bound of $|A(K)_{\mathrm{tors}}|$ for abelian varieties of bounded dimension over fields of bounded degree; as well as of the size of their isogeny classes \cite[Theorem 1.1.]{MR3815154}. In \cite{wrap131525}, the authors show that Colemann conjecture also bounds $|\Br(\overline{X})^{G_K}|$ uniformily for abelian varieties and K3 surfaces. On the other hand, there are no known relations between Shafarevich conjecture and the boundeness of the Brauer groups of K3 surfaces. 
\begin{conj}[Folklore] \label{conj finiteness isogenous}
Let $N$ be an integer and let $X/ \overline{\Q}$ be a K3 surface. Then, the number of isomorphism classes of K3 surfaces isogenous to $X$ that can be defined over extensions of degree at most $N$ can be bounded only in terms on $N$.
\end{conj}
In the following, $\ell(A_X)$ denotes the minimal number of generators of the finite abelian group $A_X$. We prove: 
\begin{thm} \label{Thm implications}
\begin{enumerate}
\item[] 
\item  Assume that Shafarevich Conjecture is true and let $N >0 $ be an integer. Then, there is a integer $M$ such that, for any K3 surface $X$ defined over a number field $K$ with $[K \colon \Q] \leq N$, which satisfies both $\rho_{\bar{X}} \geq 3 + \ell(A_X)$ and $\ell(A_X) \leq 19 - \rho(X)$, if $\Br(\overline{X})[d] \subset \Br(\overline{X})^{G_K}$ then $d \leq M$.
 \item Assume that Conjecture \ref{conj finiteness isogenous} is true and let $N >0 $ be an integer. Then, there is a integer $M$ such that, for any K3 surface $X$ defined over a number field $K$ with $[K \colon \Q] \leq N$, such that $\End(\mathrm{T}(X_\C))= \Q$ or a totally real field and $\ell(A_X) \leq 19 - \rho(X)$, if $\Br(\overline{X})[d] \subset \Br(\overline{X})^{G_K}$ then $d \leq M$.
\end{enumerate}
If one considers only level structures with $d$ and $\mathrm{discr}(\NS(X))$ coprime, one can remove the condition $\ell(A_X) > 19 - \rho(X)$ from both statements, and one can replace the condition $\rho_{\bar{X}} \geq 3 + \ell(A_X)$ with $\rho_{\bar{X}} \geq 3$ in the first statement. 
\end{thm}
Note that the condition $\ell(A_X) \leq 19 - \rho(X)$ holds automatically when $\rho(X) \leq 9$. The result can be strengthened if instead of level structures one considers the finiteness of the isotropic Brauer lines or classes.

For sake of completeness, we include in the end of the paper an alternative proof of the following finiteness:

\begin{thm} \label{finiteness thm}
Let $X/ \overline{\Q}$ be a K3 surface and suppose that $\End_{\Hdg}(T(X_\C)) = \Z$. Then, for any integer $N>0$, only finitely many isomorphism classes of K3 surfaces $Y / \overline{\Q}$ that are isogenous to $X$ can be defined over field extensions $K/ \Q$ of degree $[K \colon \Q] \leq N$. 
\end{thm}
This is well-known to the experts, and it can be proved, for instance, using the Masser-W\"{u}stholz's isogeny estimates for abelian varieties \cite{MR1217345} and the Zarhin trick for K3 surfaces invented by Charles (see \cite{MR3548531}) in complete analogy to the strategy employed in \cite{wrap131525} to prove the finiteness of CM K3 surfaces. We give an elementary proof more in line with the techniques expressed in the paper, and which only relies on the open-adelic image proved by Cadoret and Moonen in \cite{cadoret_moonen} and some classical finiteness properties of adelic algebraic groups. 
\subsection{Acknowledgment}
This work was funded by the Europen Research Council (ERC) under the European Union's Horizon 2020 research and innovation programme - grant agreement No 948066. 

I wish to thank Salvatore Floccari, Giacomo Mezzedimi, Nebojsa Pavic, and Stefan Schreieder for the many enjoyable discussions about mathematics and more.   

\subsection{Notation and convention}
All the fields in this paper are considered inside a fixed field $\C$ of the complex numbers. If $X/K$ is a K3 surface over a number field, we denote by $\overline{X}$ its base-change to $\overline{K}$, and by $X_\C$ its base-change to $\C$. By a lattice we shall always mean a free, finitely generated $\Z$-module $L$ endowed with a non-degenerate bilinear form $L \times L \rightarrow \Z$. We denote by $L^\vee$ its dual and we say that $L$ is even if $(x,x) \in 2\Z$ for every $x \in L$. For a finite abelian group $A$, we denote by $\ell(A)$ its length, that is, the minimal number of generators.
\section{Preliminaries on Hodge structures of K3 surfaces and Galois representations} \label{section 2}
In this section we set-up the notation and we recall some well-known theory of K3 surfaces. Our only contribution is Proposition \ref{prop def hodge isom}.
\subsection{Singular cohomology of K3 surfaces} 
 Let $X/ \C$ be an algebraic K3 surface and let $\mathrm{H}^2_{\text{sing}}(X, \Z)$ be its second singular cohomology group. This is a free $\Z$ module of rank $22$. The intersection pairing $$\mathrm{H}^2_{\text{sing}}(X, \Z)(1) \times \mathrm{H}^2_{\text{sing}}(X, \Z)(1) \rightarrow \mathrm{H}_{\text{sing}}^4(X, \Z)(2) \cong \Z$$ endows $\mathrm{H}^2_{\text{sing}}(X, \Z)(1)$ with a non-degenerate quadratic form, and from Poincare duality one deduces that $\mathrm{H}^2_{\text{sing}}(X, \Z)(1)$ is \textit{unimodular}, i.e., that $|\mathrm{discr}(\mathrm{H}^2_{\text{sing}}(X, \Z)(1))| =1$. The isomorphism class of this lattice does not depend on the chosen $X$, and one can also prove that $\mathrm{H}^2_{\text{sing}}(X, \Z)(1)) \cong U^{3} \oplus E_8(-1)^2,$ where $U$ denotes the hyperbolic lattice and $E_8(-1)$ the unique even, negative definite, unimodular lattice of rank $8$. There is a natural identification between $\Pic(X)$ and the N\'{e}ron-Severi lattice $\NS(X) \subset \mathrm{H}^2_{\text{sing}}(X, \Z)(1)$. The orthogonal complement of $\NS(X)$ is the transcendental lattice $\mathrm{T}(X) \subset \mathrm{H}^2_{\text{sing}}(X, \Z)(1).$ This is an integral Hodge structure of weight $0$ with weights $(1,-1), (0,0), (-1,1)$ and with $\dim_{\C}(T^{1,-1}_X) = 1$. The Hodge structure $\mathrm{T}(X,\Q) = \mathrm{T}(X) \otimes_\Z \Q$ is irreducible whenever $X$ is projective, and $\mathrm{T}(X)$ can also be characterized as the biggest, primitive sub-Hodge structure of $\mathrm{H}^2(X,\Z)(1)$ whose complexification contains the period of $X$. 
\begin{defi}
The finite group $A_X = \NS(X)^\vee / \NS(X)$ is called the \textit{discriminant form} of $X$.
\end{defi}
Here $L^\vee$ denotes the dual of a lattice $L$. The group $A_X$ carries a natural quadratic form $q \colon A_X \rightarrow \Q/ 2\Z$ given by $q(\ell + \NS(X) ) := (\ell, \ell) + 2\Z$ (this is well defined since $\NS(X)$ is even). Similarly, one has a quadratic form $q'$ on $\mathrm{T}(X)^\vee / \mathrm{T}(X)$, and a result of Nikulin \cite[Proposition 1.6.1]{MR525944} shows that there is a natural isomorphism
 \begin{equation} \label{Nikulin's isomorphism}
     (A_X, q) \cong (\mathrm{T}(X)^\vee / \mathrm{T}(X), -q'),
 \end{equation}
 due to the fact that $\NS(X)$ and $\mathrm{T}(X)$ are orthogonal complements in the unimodular lattice $\mathrm{H}^2_{\text{sing}}(X, \Z)$. Note that $|A_X| = |\mathrm{disc}(\mathrm{T}(X))| = |\mathrm{disc}(\NS(X))|.$ For the next proposition, see \cite[Corollary 1.5.2]{MR525944}.
\begin{prop} \label{Nikulin extension}
Let $X,Y$ be complex K3 surfaces, and let $f_t \colon \mathrm{T}(X) \rightarrow \mathrm{T}(Y)$ and $f_n \colon \NS(X) \rightarrow \NS(Y)$ be isometries such that the two induced maps $A_X \rightarrow A_Y$ coincide. Then there exists a unique isometry $\mathrm{H}^2_{\text{sing}}(X, \Z)(1) \rightarrow \mathrm{H}^2_{\text{sing}}(Y, \Z)(1)$ that restrict to $f_t$ and $f_n$ on $\mathrm{T}(X)$ and $\NS(X)$ respectively. 
\end{prop}
We recall that the Grothendieck-Brauer group $\Br(X) = \mathrm{H}^2_{\et}(X, \Oo_X^\times)$ of $X$ can be written via the exponential sequence as $\Br(X) \cong \mathrm{T}^\vee (X)_X \otimes_\Z \Q/ \Z \cong \Hom(\mathrm{T}(X), \Q / \Z).$ In particular, $\Br(X)$ is a torsion abelian group of the form $(\Q/ \Z)^{22 - \rho_X}.$
\subsection{\`{E}tale cohomology and Galois action}
Suppose now that $X$ is defined over a number field $K \subset \C$ and let $\bar{K}$ be the algebraic closure of $K$ in $\C.$ Let $\mathrm{H}^2_{\et}(\bar{X}, \widehat{\Z})(1)$ be the Tate twist of the second \et ale cohomology group of $\bar{X} : = X \times_K \Spec(\bar{K})$. The inclusion $\bar{K} \subset \C$ induces a comparison isomorphism
\begin{equation} \label{Comparison isomorphism}
\mathrm{H}^2_{\et}(\bar{X}, \widehat{\Z})(1) \cong \mathrm{H}^2_{\text{sing}}(X_\C, \Z)(1) \otimes \widehat{\Z},
\end{equation}
so that we can view $\mathrm{H}^2_{\text{sing}}(X, \Z)(1) \subset \mathrm{H}^2_{\et}(\bar{X}, \widehat{\Z})(1)$ via the map $v \mapsto v \otimes 1.$ We denote the orthogonal complement of $\NS(\bar{X})$ in $\mathrm{H}^2_{\et}(\bar{X}, \widehat{\Z})(1)$ by $T(\overline{X}, \widehat{\Z})$. As the notation suggests, via the isomorphism \ref{Comparison isomorphism}, one has $$T(\overline{X}, \widehat{\Z}) \cong \mathrm{T}(X_\C) \otimes \widehat{\Z}$$ and a natural isomorphism $$\Br(\bar{X}) \cong \mathrm{T}(\overline{X}, \widehat{\Z})^\vee \otimes \Q / \Z.$$ 
To $X / K$ there is associated a (continuous) Galois representation $$\rho_X \colon G_K \rightarrow \Gl(\mathrm{H}_{\text{sing}}^2(X, \Z))(\widehat{\Z}).$$ 
Since the Galois action respects Poincar\'{e} duality, one sees that $\rho_X$ acts by $\widehat{\Z}$-isometries. The action on the N\'{e}ron-Severi lattice is not very interesting due to the following well-known phenomenon: 
\begin{lemma} \label{lemma Neron-Severi}
There exists a constant $C_n$ with the following property: for any K3 surface $X/K$ as above, there is a field extension $K'/K$ of degree $[K \colon K'] \leq C_n$ such that $G_{K'}$ acts trivially on $\NS(\overline{X})$.
\end{lemma}
\begin{proof}
Since $\NS(\overline{X})$ is a free $\Z$-module and $\rho_X$ is continuous, the image of $\rho_X$ must be a finite subgroup of $\Gl(\NS(\overline{X}))(\Z)$. The statement follows from the classical fact that any such subgroup embedds into $\Gl(\NS(\overline{X}))(\Z / 3)$, and that for K3 surfaces the rank of the N\'{e}ron-Severi group is always smaller than $20$.
\end{proof}
Let $\MT_X \subset \mathrm{GL}(\mathrm{T}(X, \Q))$ be the Mumford-Tate group of $X$; Mumford-Tate groups of K3 surfaces are completely classified by Zarhin, and they only depends on the field $E_X \coloneqq \End_{\Hdg}(\mathrm{T}(X,\Q))$ and on its action on the quadratic space $\mathrm{T}(X, \Q)$. When $E_X = \Q$, one has $\MT_X = \mathrm{SO}(\mathrm{T}(X, \Q))$. Since the Mumford conjecture holds for K3 surfaces, see \cite{zbMATH01181394} and \cite{zbMATH00931955}, one knows due to a result of Serre (see \cite[Proposition 6.14]{zbMATH00039403}) that that there is a finite extension $K'/K$ such that $\rho_X(G_K) \subset \MT_X(\A_f)$. The following important theorem is due to Cadoret and Moonen: 
\begin{thm}[\cite{cadoret_moonen}] \label{integral mt conj}
Let $X$ be a K3 surface over a number field $K$ and assume that $\rho_X(G_K) \subset \MT_X(\A_f)$. Then, $\rho_X(G_K) \subset \MT_X(\A_f)$ is a compact open subgroup for the adelic topology. 
\end{thm}
Note that the same result does not hold in general for abelian varieties for which the Mumford-Tate conjecture is known, because their Mumford-Tate groups are not always \textit{Hodge-maximal}. 
\subsection{Galois action on Hodge classes} Let $E_X \coloneqq \End_{\Hdg}(\mathrm{T}(X,\Q))$ be the algebra of Hodge endomorphisms of $\mathrm{T}(X,\Q)$ and let $U_X \subset E_X$ be the subgroup of Hodge isometries. One knows \cite{MR697317} that $E_X$ is always a field, either totally real, or CM. One also knows that $E_X \subset \mathrm{H}^4(X \times X, \Q(2))$ consists of motivated classes in the sense of Andr\'{e} or absolute Hodge classes in the sense of Deligne, and this is also true for $\mathrm{Hom}_{\text{Hdg}}(\mathrm{T}(X, \Q), \mathrm{T}(Y,\Q)) \subset \mathrm{H}_{\mathrm{sing}}^4(X \times Y, \Q(2))$ for another K3 surface $Y$. In both situations, for $\sigma \in \Aut(\C)$ and any $f \in \mathrm{Hom}_{\text{Hdg}}(\mathrm{T}(X, \Q), \mathrm{T}(Y,\Q))$ there corresponds $f^\sigma \in \mathrm{Hom}_{\text{Hdg}}(\mathrm{T}(X^\sigma,\Q), \mathrm{T}(Y^\sigma,\Q))$, in such a way that the following diagram commutes 
\begin{equation} \label{square galois}
\begin{tikzcd} 
\mathrm{T}(X^\sigma,\A_f) \arrow{r}{f^\sigma} \arrow{d}{\sigma^*_X}
& \mathrm{T}(Y^\sigma,\A_f) \arrow{d}{\sigma_{Y}^*} \\
\mathrm{T}(X,\A_f) \arrow{r}{f}
& \mathrm{T}(Y,\A_f).
\end{tikzcd}
\end{equation}
Here, $\sigma^*_X$ denotes the natural pullback map in \et ale cohomology induced by the morphism of schemes $\sigma_X \colon X \rightarrow  X^\sigma$, and similarly for $Y$. 
\begin{defi} \label{definition f}
Assume that $X$ and $Y$ can be defined over a number field $K$ and let $f \in \mathrm{Hom}_{\text{Hdg}}(\mathrm{T}(X,\Q), \mathrm{T}(Y,\Q)).$ We say that $f$ is defined over $K$ if $f^{\sigma} = f$ for every $\sigma \in G_K$.
\end{defi}
Note that by \eqref{square galois}, if $X=Y$, this happens if and only if $f \in E_X$ commutes with the action of $G_K$. There is a straightforward way to understand whether an element of $E_X$ is defined over $K.$ One first notes that if $X/ \C$ is a complex K3 surface, there is a natural map 
\begin{equation} \label{Inclusion fields}
E_X \rightarrow \End(T^{1,-1}(X_\C)) = \C
\end{equation}
which can be easily proved to be an embedding of fields, due to the fact that $\mathrm{T}(X, \Q)$ is irreducible whenever $X$ is projective. It follows from \cite{VALLONI2021107772} that $E_X$ is defined over $K$ if and only if, for one embedding $\sigma \colon K \rightarrow \C$, and hence any, the image of \eqref{Inclusion fields} is contained in $\sigma(K)$. Since $[E_X \colon \Q] \leq 20$ always, we obtain the following corollary:
\begin{cor} \label{corollary field defi hodge end}
For any K3 surface $X/K$ there exists a finite extension $K'/K$ of degree $[K' \colon K] \leq 20$ over which all the elements of $E_X$ are defined. 
\end{cor}
We conclude the section with a uniform bound similar to Lemma \ref{lemma Neron-Severi} and to the Corollary above.  

\begin{prop} \label{prop def hodge isom}
There exists a universal constant $C_h > 0 $ with the following property. Suppose that $X$ and $Y$ are K3 surfaces, and assume that they are defined over a number field $K$. Let $f \colon \mathrm{T}(Y_\C, \Q) \rightarrow \mathrm{T}(X_\C, \Q)$ be a Hogde isometry. Then, there exists a field extension $K'/K$ of degree $[K' \colon K] \leq C_h$ such that $f$ is defined over $K'$. 
\end{prop}
\begin{proof}
Up to an extension of degree bounded by $20$, we can assume that $E_X$ is defined over $K$. Since both $Y$ and $X$ are defined over $K$, there is a natural $G_K$ action on the set $\text{Isom}_{\Hdg}(\mathrm{T}(Y_\C, \Q), \mathrm{T}(X_\C, \Q)).$ The map $G_K \ni \sigma \mapsto f^{\sigma} \circ f^{-1} \in U_X$ is a continuous $1-$cocycle of $G_K$ with values in $U_X$ (where $G_K$ has its natural profinite topology and $U_X$ the discrete one). Since the Galois action on $U_X$ is trivial, it follows that the map $\sigma \mapsto f^{\sigma} \circ f^{-1}$ is a continuous morphism of groups, and hence it must factorize as $G_{K} \rightarrow (U_{X})_\text{tors} = \mu(E_X) \subset U_X,$ where $\mu(E_X) $ is the group of the roots of unity in $E_X$. The order $ | \mu(E_X) | $ is universally bounded for K3 surfaces, since $\mathrm{T}(X_\C,\Q)$ is always a vector space over $\Q(\mu(E_X))$, and $20 \geq \dim_{\Q} \mathrm{T}(X_\C,\Q) \geq \phi ( | \mu(E_X) |),$ where $\phi$ denotes the Euler totient function.
\end{proof}
As the proof shows, in the general case where the only roots of unity in $E_X$ are $\{ \pm 1 \}$, the extension can be chosen to be quadratic. 

\section{Isogeny class and lattices} \label{Section 3}
In this section we study the isogeny class of a complex projective K3 surface $X$ in terms of lattices inside $T(X)_\Q$. 
\subsection{Basics} 
Let $X$ be a complex projective K3 surface. The marked isogeny class of $X$, denoted by $\mathcal{H}^{m}_X$, consists of all the pairs $(Y,f)$ where $Y \in \mathcal{H}_X$ and $f \colon \mathrm{T}(Y,\Q) \rightarrow \mathrm{T}(X,\Q)$ is a Hodge isometry, up to the natural notion of isomorphism.  

Consider the set $\mathcal{L}$ of all the free $\Z$-modules $L \subset \mathrm{T}(X,\Q)$ of maximal rank such that the restriction of the quadratic form to $L$ is even, and hence also integral. In this way, we obtain a map $ \mathcal{H}^{m}_X \rightarrow \mathcal{L}$ given by $(Y,f) \mapsto f(\mathrm{T}(Y))$ and we denote by $\mathcal{L}_{K3} \subset \mathcal{L}$ its image. The map $ \mathcal{H}^{m}_X \rightarrow  \mathcal{L}_{K3}$ descends to a map $\mathcal{T} \colon  \mathcal{H}_X \rightarrow U_X \backslash \mathcal{L}_{K3}$ where $U_X \backslash \mathcal{L}_{K3}$ denotes the set of orbits in $\mathcal{L}_{K3}$ under the natural left action of $U_X$, that is, the set of $U_X$-classes of lattices. Elements in the same fibre of $\mathcal{T}$ have isometric integral transcendental lattices and it follows that $\mathcal{T} $ has finite fibres. Note that if $E_X$ is totally real one has $U_X = \{ \pm 1 \}$, so that up to Fourier-Mukai partners, the correspondence between $\mathcal{H}_X$ and $\mathcal{L}_{K3}$ is one-to-one. 
\begin{prop} \label{proposition L_K3}
$\mathcal{L}_{K3}$ consists of all the lattices in $\mathrm{T}(X,\Q)$ that can be embedded primitively into the K3 lattice. Moreover, if $L \in \mathcal{L}$, and if there exists $L' \in \mathcal{L}_{K3}$ such that $L' \subset L$, then $L \in \mathcal{L}_{K3}$.
\end{prop}

\begin{proof}
To prove the first statement, let $L \subset \mathrm{T}(X,\Q)$ and fix a primitive embedding of $L$ into $\Lambda_{K3}$. The lattice $L$ is endowed with a natural Hodge structure, determined by the fact that the inclusion $L \subset \mathrm{T}(X,\Q)$ respects the Hodge decomposition. Endowing $\Lambda_{K3}$ with such Hodge structure, by the surjectivity of the period map, we obtain a K3 surface $Y$ endowed with a Hodge isometry $f \colon \mathrm{T}(Y,\Q) \rightarrow \mathrm{T}(X,\Q)$ such that $f(\mathrm{T}(Y)) = L$.

 For the other statement, we need to show that if a lattice $L$ embedds primitively into $\Lambda_{K3}$, then every even overlattice $L \subset L'$ of finite index embedds primitively into $\Lambda_{K3}$ as well. This follows from Nikulin's theory: a primitive embedding of $L$ into $\Lambda_{K3}$ with orthogonal complement isomorphic to $N$ is determined by an isomorphism of finite quadratic forms $A_L \cong -A_N$. The overlattice $L'$ is determined uniquely by an isotropic subgroup $H \subset A_L$ and there is a natural isomorphism of finite quadratic forms $A_{L'} \cong H^\perp/H$ (\cite[Proposition 1.4.1]{MR525944}). Under the identification $A_L \cong -A_N$, we see that $H$ also determines an overlattice $N'$ of $N$, and that there is an isomorphism of quadratic forms $A_{L'} \cong - A_{N'}$. This shows that there is a primitive embedding of $L'$ into a unimodular, even lattice of signature type $(3,19)$ with orthogonal complement isomorphic to $N'$. Since the K3 lattice is the unique unimodular even lattice of signature $(3,19)$, we conclude. 
\end{proof}
In particular, when $\rho(X) \geq 12$, the map $ \mathcal{T} \colon \mathcal{H}_X \rightarrow U_X \backslash \mathcal{L}$ is a bijection because every even lattice of rank smaller then $10$ embedds primitively into the K3 lattice, and because, in this case, $\mathrm{T}(X)$ determines $X$ uniquely. 

In the next, we make some considerations that often appear in the theory of Shimura varieties and moduli spaces of K3 surfaces. An adelic point $g \in \mathrm{SO}_X(\A_f)$ of $\mathrm{SO}_X$ is given by a collection of elements $g_p \in \mathrm{SO}_X(\Q_p)$ indexed by each prime $p$ and such that, for almost all $p$, one has $$g_p(\mathrm{T}(X) \otimes \Z_p) = \mathrm{T}(X) \otimes \Z_p,$$ with the equality occurring in $\mathrm{T}(X) \otimes \Q_p$. It is well-known that for any $g \in \mathrm{SO}_X(\A_f)$ and any lattice $L \in \mathcal{L}$, there exists a unique lattice $gL \in \mathcal{L}$ such that, for each prime $p$, one has $gL \otimes \Z_p = g_p(L \otimes \Z_p )$. In this way one lets $\mathrm{SO}_X(\A_f)$ act on $\mathcal{L}$.
\begin{lemma} \label{lemma SO respects L_K3}
The action of $\mathrm{SO}_X(\A_f)$ on $\mathcal{L}$ respects $\mathcal{L}_{K3} \subset \mathcal{L}$.
\end{lemma}
\begin{proof}
Let $g \in \mathrm{SO}_X(\A_f)$, let $L \in \mathcal{L}_{K3}$ and consider $gL \in \mathcal{L}$. The lattices $L$ and $gL$ are locally isometric. By \cite[Theorem 1.12.2]{MR525944} of Nikulin, an even lattice $L$ embedds primitively into the lattice $\Lambda_{K3}$ if and only if some local conditions are satisfied at every prime of $\Z$; from this it follows that also $gL$ embedds primitively in $\Lambda_{K3}$.
\end{proof}

Thus, if $L \in \mathcal{L}_{K3}$, the whole orbit $\mathrm{SO}_X(\A_f) \cdot L$ lives inside $\mathcal{L}_{K3}$. The set $\mathrm{SO}_X(\A_f) \cdot L$ is called the genus of $L$, and it is a classical result that $\mathrm{SO}_X(\A_f) \cdot L$ can be written as a finite union of classes, sets of the form $\mathrm{SO}_X(\Q) \cdot L'$ for some lattices $L' \in \mathcal{L}_{K3}$. If $L$ is an indefinite lattice, then in most of the situations one has that its genus consists only of one class. This for example happens when $\ell(A_L) \leq \mathrm{rank}(L) -2$ and $\mathrm{rank}(L) \geq 3$, see \cite[Theorem 1.13.2]{MR525944}. Nevertheless, the genus may be in general non-trivial. The following notion is a particular case of one introduced by Shimura, see \cite[Chapter 29]{MR2665139}.
\begin{defi}
A lattice $L \in \mathcal{L}$ is said to be \textit{maximal} if there is no other lattice $L' \in \mathcal{L}$ such that $L' \subsetneq L$. We denote the set of maximal lattices as $\mathcal{L}^{\mathrm{max}} \subset \mathcal{L}$ and we denote by $\mathcal{H}_X^{\mathrm{max}}$ the subset of the isogeny class consisting of K3 surfaces $Y$ such that $\mathrm{T}(Y)$ is maximal.   
\end{defi}
Clearly, the action of $\mathrm{SO}_X(\A_f)$ on the set $\mathcal{L}$ also respects $\mathcal{L}^{\mathrm{max}}$. The following is \cite[Lemma 29.2, Lemma 29.9]{MR2665139}.
\begin{prop}
Any lattice $L \in \mathcal{L}$ is contained in some maximal lattice. Moreover, if $L, L' \in \mathcal{L}^{\mathrm{max}}$, then there exists $g \in \mathrm{SO}_X(\A_f)$ such that $L' = gL$.
\end{prop}
\begin{rmk} \label{remark on class and genus}
The result above says that $\mathcal{L}^{\mathrm{max}}$ consists of a single genus with respect to $\mathrm{SO}_X$.
\end{rmk}
As a corollary to the previous proposition and of Lemma \ref{lemma SO respects L_K3}, we obtain
\begin{cor}
One has that $\mathcal{L}^{\mathrm{max}} \subset \mathcal{L}_{K3}$. In particular, there are always maximal K3 surfaces in each isogeny class. 
\end{cor}

\subsection{Constructing elements of $\mathcal{H}_X$} \label{subsection constructing from Brauer}
If $A/ \C$ is an abelian variety, one can construct all of its isogeny class by taking quotients of $A$ by its finite subgroups. For K3 surfaces, one does not have similar straightforward ways to build isogenies. On the other hand, due to Proposition \ref{proposition L_K3}, in order to construct isogenous K3 surfaces one only needs a way to construct elements in $\mathcal{L}_{K3}$. It is reasonable to believe that the role of torsion points on abelian varieties should be replaced by elements of $\mathrm{T}(X) \otimes \Q/ \Z$ or, similarly, by Brauer classes. One natural way to construct elements of $\mathcal{L}$ via Brauer classes on $X$ is as follows. Let $B \subset \Br(X)$ be a finite subgroup. Using that $\Br(X) \cong \Hom(\mathrm{T}(X), \Q/ \Z)$ one defines $\ker(B) \subset \mathrm{T}(X)$ as the obvious sublattice. If $\ker(B)$ embedds primitively into the K3 lattice (this is automatically true when $\mathrm{rank}(\mathrm{T}(X)) \leq 10$) then there is at least a K3 surface $Y$ whose transcendental lattice is Hodge isometric to $\ker(B)$, i.e., the K3 surface $Y$ has a Hodge isometry $f \colon \mathrm{T}(Y, \Q) \rightarrow \mathrm{T}(X, \Q)$ such that $f(\mathrm{T}(Y)) = \ker(B)$. The problem in this situation is twofold: first, this will not work generically for K3 surfaces of small Picard rank, and second, even if $\ker(B)$ does embedd primitively into the K3 lattice, the associated K3 surface $Y$ will satisfy $|\mathrm{discr}(\NS(Y))| = |B|^2 \cdot |\mathrm{discr}(\NS(X))|$, and in particular one never obtains transcendental lattices that are in the same genus of $\mathrm{T}(X)$. 

Some examples with $B= \Z / 2 \Z$ have a geometric interpretation, see \cite{zbMATH02217905} and \cite{MR3616011}. Most examples of this situation come from the theory of moduli spaces of sheaves on K3 surfaces developed by Mukai \cite{MR893604}, as we shall now briefly explain. Let $X$ be a complex K3 surface and let $v \in \mathrm{H}^*(X, \Z)$ be a primitive and isotropic Mukai vector. Let $X = M(v)$ be the moduli space of stable sheaves on $X$ whose Mukai vector is $v$, where `stable' is intended with respect to a generic polarization. Then, if $M(v)$ is non-empty, it is a K3 surface. Mukai shows that on $X \times Y$ there is a quasi-universal family, and a slight modification of its Chern class induces an isogeny $f \colon \mathrm{T}(Y) \rightarrow \mathrm{T}(X)$ such that $f(\mathrm{T}(Y)) \subset \mathrm{T}(X)$. In fact, there is a Brauer class $\alpha \in \Br(X) \cong \Hom(\mathrm{T}(X), \Q / \Z)$ such that $\mathrm{ker}(\alpha) = f(\mathrm{T}(Y))$, and in particular the quotient $\mathrm{T}(X)/ f(\mathrm{T}(Y))$ is cyclic. This Brauer class is computed in \cite{MR1902629} and can be seen as an obstruction to the existence of a universal family on $X \times Y$, that is, to $X$ being a fine moduli space: if $\alpha = 0$ then $\mathrm{T}(X) \cong \mathrm{T}(Y)$, the moduli space $X$ is fine, and $X$ and $Y$ are Fourier-Mukai partners. In particular, if $Y$ is maximal, then every such $X$ is a fine moduli space. One sees that via this construction only finitely many lattices in $\mathrm{T}(Y, \Q)$ can be obtained. \cite[Proposition 1.4.1]{MR525944}

We recall some basic lattice theory concerning this situation. For an even lattice $L$ we denote by $A_L \cong L^\vee/L$ its discriminant form, and by $q \colon A_L \rightarrow \Q/ 2 \Z$ the associated quadratic form. Note that $A_L$ is also endowed with a non-denerate pairing $A_L \times A_L \rightarrow \Q/ \Z$. Given an inclusion of lattices $N \subset L$ of finite cokernel, there is a flag
\begin{equation} \label{filtration}
    N \subset L \subset L^\vee \subset N^\vee,
\end{equation}
and $L/N \subset A_N$ is an isotropic subgroup of $A_N$, i.e., the quadratic form $q \colon A_N \rightarrow \Q / 2 \Z$ is zero when restricted to $L/N$. In fact, even overlattices of $N$ corresponds one-to-one to isotropic subgroups of $A_N$, see \cite[Proposition 1.4.1]{MR525944}. One also has a natural isomorphism 
\begin{equation} \label{iso graded}
   N^\vee/L^\vee \cong \Hom(L/N, \Q/ \Z),
\end{equation}
and that $(L/N)^\perp = L^\vee/N$, where the orthogonal is taken with respect to the natural pairing $A_N \times A_N \rightarrow \Q / \Z$, and in particular $(L/N)^\perp/ (L/N) \cong A_L$. In case of $N = \ker(B)$ and $L = \mathrm{T}(X)$, one readily sees that $B \cong \Hom(\mathrm{T}(X)/ \ker(B), \Q/ \Z)$. 
\begin{defi} \label{lower isogeny definition}
The lower isogeny class $\mathcal{H}^\mathrm{low}_X$ of $X$ consists of all the $Y \in \mathcal{H}_X$ for which there exists a Hodge isometry $f \colon \mathrm{T}(Y, \Q) \rightarrow \mathrm{T}(X, \Q)$ such that $f(\mathrm{T}(Y)) = \mathrm{T}(X)$.
\end{defi}

\subsection{Kneser neighbors} \label{subsection Kneser}
Kneser method of neighboring lattices has two important advantages with respect to the situation before: it works for every K3 surface, and it allows one to construct lattices in the same genus of $\mathrm{T}(X)$. The original reference is \cite{zbMATH03127648}, for the following treatment we follow Chapter 3 of \cite{zbMATH07417850}. We give a self contained proof of most of the results using Nikulin's language. Let $L$ be an even lattice, and let $d > 1$ be an integer that is coprime to $\mathrm{discr}(L)$. Denote by $\tilde{q} \colon L \rightarrow \Z$ the quadratic form given by $2\tilde{q}(x) = (x,x)$ where the brackets denote the bilinear form of $L$. Another even lattice $L' \subset L \otimes \Q$ is called a $d$-neighbor of $L$ if $$L / L \cap L' \cong L' / L \cap L' \cong \Z / d \Z. $$
The original definition of Kneser only contemplated the case $d=2$, but his results can be generalized to any integer $d$. 
\begin{thm}[Kneser] \label{Kneser thm}
If $L$ and $L'$ are $d$-neighbors, then $L$ and $L'$ belong to the same genus. 
\end{thm}
\begin{proof}
Let $M \coloneqq L \cap L'$, and let $p$ be a prime that divides $d$. Then, the lattice $L_p = L \otimes \Z_p$ is unimodular (and even, if $p=2$) and the lattice $L'_p = L' \otimes \Z_p$ is integral (and even, if $p=2$). We have a short exact sequence $0 \rightarrow L_p/M_p \rightarrow A_{M,p} \rightarrow M_p^\vee/L_p \rightarrow 0$ where we denote $A_{M,p} = A_{M} \otimes \Z_p = A_{M \otimes \Z_p}$. The overlattice $L'_p$ of $M_p$ is given by an isotropic subgroup $L'_p/M_p \subset A_{M,p}$ which is transverse to $L_p/M_p$. In particular $A_{M,p} \cong L_p/M_p \oplus L'_p/M_p$, with both summands isomorphic to $\Z/p^n\Z$ where $p^n$ is the maximal power of $p$ that divides $d$. An element $ [v] \in L'_p/M_p$ orthogonal to $L_p/M_p$ is represented by an element $v \in L'_p$ such that $(v,L_p) \subset \Z_p$; since the lattice $L_p$ is unimodular, it is also maximal, and we conclude that $v \in L_p$. By Nikulin \cite[Proposition 1.4.1]{MR525944}, the discriminant form of $L'_p$ is given by $(L'_p/M_p)^{\perp} / (L'_p/M_p) = 0$. In particular, $L_p$ and $L_p'$ are two $p$-adic lattices (even, if $p=2$) with same rank and discriminant forms. Moreover, their discriminants satisfy $\mathrm{discr}(L_p) = p^{2n} \mathrm{discr}(M_p) = \mathrm{discr}(L'_p)$, and one concludes that $L_p$ and $L'_p$ are isometric due to \cite[Corollary 1.9.3]{MR525944}. For primes $p$ that do not divide $d$, one has $L_p = L'_p$, and we conclude the proof. 
\end{proof}
One constructs $d$-neighbors of a lattice $L$ using isotropic lines in $L/ dL$. Since $d$ is coprime to $\mathrm{discr}(L)$, the induced quadratic form $\tilde{q} \colon L/dL \rightarrow \Z / d \Z$ is non-degenerate. A $d$-isotropic line $\ell$ is by definition a cyclic subgroup $\Z / d \Z \subset L/ dL$ which is isotropic with respect to this quadratic form. One has the \textit{line map}
$$ \{d\text{-neighbors of} \,\, L \} \rightarrow \{d\text{-isotropic lines in }L/dL \} $$
that associated to any $d$-neighbor $L'$ of $L$ the line $\ell = \im(dL' \rightarrow L \rightarrow L/dL)$.

\begin{prop} \label{prop line map}
The line map is well defined, and it is a bijection. 
\end{prop}

\begin{proof}
Let $\ell \subset L/dL$ be an isotropic line generated by $v \in L$. We begin by showing that we can choose $v$ such that $v^2 \in 2d^2 \Z$. For any $w \in L$, we compute $(v+dw)^2 = 2\tilde{q}(v+dw) = 2 \tilde{q}(v) + 2d^2\tilde{q}(w) + 2d(v,w)$. Write $\tilde{q}(v) = dk$ for $k \in \Z$ and consider the equation $k + (v,w) = 0 \mod d.$ Since the bilinear pairing on $L/dL$ is perfect, we deduce that we can always solve this equation, and hence that we can always find a $w$ such that $(v+dw)^2 \in p^2 \Z$. Choose such $v$ and consider the map $h \colon L \rightarrow \Z/d\Z$ given by pairing with it, and let $M=\ker(h)$. Note that $M^\vee$ is generated by $1/d v$ and $L$, and we want to show that the lattice $L'$ generated by $1/d v$ and $M$ is a $d$-neighbor of $L$. By the fact that $v$ generates a group isomorphic to $\Z / d \Z$ one deduces that $L \cap L' = M$, and since $v^2 \in 2d^2 \Z$ one easily sees that $L'$ is even. It is straightforward to check that this is the inverse of the line map.  
\end{proof}
Note that since $d$ and $\mathrm{discr}(L)$ are coprime then $L \otimes \Z/ d\Z \cong \Hom(L, \Z/d) \cong (d^{-1} L)/L.$ 
\begin{lemma}  \label{equation isomorphism isotropic line}
Consider an isotropic line $\ell \subset L \otimes \Z/d\Z$ as a subgroup $\ell \subset \Hom(L, \Z/d)$. Then, $\ker(\ell) = L \cap L'$ where $L'$ is the $d$-neighbor associated to $\ell$. Moreover, one has a natural identification $\Hom(\ell, \Q/ \Z) \cong L/M$.
\end{lemma}
\begin{proof}
The first part follows from the fact that one can choose a generator of $\ell$ to satisfy $v^2 \in 2d^2 \Z$ as explained in the previous proof. The second part is \eqref{iso graded}.
\end{proof}
\begin{lemma} \label{lemma transverse}
Let $L$ and $L'$ be $d$-neighbors, and let $M = L \cap L'$. Let $H \cong \Z/d\Z \subset A$ be the isotropic subgroup associated to the overlattice $L$ and define $H'$ similarly. Then, $H'$ is the unique isotropic subgroup $H' \cong \Z/d\Z  \subset A_M$ with $H \cap H' = 0$.
\end{lemma}

\begin{proof}
Due to cardinality reason, if another group $H''$ that satisfies the properties in the statement existed, then $H'' \subset H \oplus H'$ necessarily. So one can assume that $H''$ is generated by $a+b$ with $a \in H$ and $b \in H'$. One has $q(a+b) = 2(a,b) \in \Q / 2 \Z$ and one notes that this is never $0$ due to the perfect pairing from the lemma above, unless $a=0$ or $b=0$. 
\end{proof}
We remark that the unicity above fails in general when $L/M$ is not cyclic, that is, the fact that $L/M$ is cyclic is fundamental for the notion of Kneser neighbor and this situation cannot be generalized. 
\begin{defi}
Let $L$ be a lattice. A sublattice $M \subset L$ is said to be $d$-split if $d$ is corpime to $\mathrm{discr}(L)$, the quotient $L/M \cong \Z/d\Z$ and the isotropic subgroup $L/M \subset A_M$ has a unique transversal as in Lemma \ref{lemma transverse}. 
\end{defi}

In this way, one obtains a one-to-one correspondence between $d$-split sublattices, $d$-neighbors, and $d$-isotropic lines: to any $d$-isotropic line $\ell$, one associates the $d$-split sublattice $M \coloneqq \ker(\ell)$. Let $H' \subset A_M$ be the unique isotropic subgroup from Lemma \ref{lemma transverse}. This determines a unique overlattice $L'$ of $M$, and one readily shows that $L'$ is a $d$-neighbor of $L$. Finally, the isotropic $\ell = \im(dL' \rightarrow L \rightarrow L/dL)$ is the one we have started with. As explained in the proof of \ref{prop line map}, one can also construct $L'$ directly from $\ell$.
\begin{defi} \label{defi d-neighbor K3}
A K3 surface $Y \in \mathcal{H}_X$ is a $d$-neighbor of $X$ if $d$ is coprime to $\mathrm{discr}(\NS(X))$ and if there exists a Hodge isometry $f \colon \mathrm{T}(Y, \Q) \rightarrow \mathrm{T}(X, \Q)$ such that $f(\mathrm{T}(Y))$ is a $d$-neighbor of $\mathrm{T}(X)$.
\end{defi}
Due to \ref{proposition L_K3} and \ref{Kneser thm} we conclude that neighboring K3 surfaces can be constructed from isotropic lines in $\mathrm{T}(X) / d \mathrm{T}(X) \cong \Br(X)[d]$, and if $E_X$ is totally real, then the isomorphism classes of transcendental lattices of K3 surfaces that are $d$-neighbors of $X$ correspond one-to-one to isotropic lines in $\mathrm{T}(X) / d \mathrm{T}(X)$. 

\begin{rmk} \label{lang-weil}
As a consequence of the Lang-Weil estimates, one sees that if $\mathrm{rank}(L) \geq 3$ then $d \rightarrow \infty$ implies that $|\{d\text{-isotropic lines} \}| \rightarrow \infty$ too. 
\end{rmk}

We need the following lemma:

\begin{lemma} \label{Lemma on iso of discriminant forms}
Let $L$ be an even lattice and let $L'$ be a lattice in $L \otimes \Q$ with discriminant forms $A$ and $A'$ respectively. There is a natural isomorphism $A \cong A'$ whenever one of the following condition is satisfied: 
\begin{enumerate}
    \item $L,L' \in \mathcal{L}^{\mathrm{max}}$;
    \item $L,L'$ belong to the same genus and $[L \colon L \cap L']$ and $\mathrm{discr}(L)$ are coprime.
\end{enumerate}
\end{lemma}
 \begin{proof}
 Let $L \coloneqq L_1 \cap L_2$. Since $\mathrm{discr}(L_1) = \mathrm{discr}(L_2)$ because $L_1$ and $L_2$ are in the same genus, we see that $$[L_1 \colon L]^2 = \frac{\mathrm{discr}(L)}{\mathrm{discr}(L_1)} = \frac{\mathrm{discr}(L)}{\mathrm{discr}(L_2)} =  [L_2 \colon L]^2, $$
 i.e., that $[L_1 \colon L] = [L_2 \colon L]$. Let $A$ be the finite quadratic form associated to $L$, and let $H_1, H_2 \subset A$ be the isotropic subgroups that determine $L_1$ and $L_2$ respectively. Since $L = L_1 \cap L_2$ by definition, we see that $H_1 \cap H_2 = 0$, and hence that $H_1 \oplus H_2 \subset A$. Moreover, we also have natural isomorphisms $A_1 \cong H_1^{\perp} / H_1$ and $A_2 \cong H_2^{\perp}/H_2$. 
 
 \begin{enumerate}
     \item To prove point (1), we want to show that the map $ H_1^{\perp} / H_1 \rightarrow A/(H_1 \oplus H_2)$ is injective, which is equivalent to $H_1^\perp \cap H_2 = 0$. This is in turn equivalent to $L_1^\vee \cap L_2 = L$. To show this last equality, assume that $v \in L_1^\vee \cap L_2$, and consider the lattice $\tilde{L}_1$ generated by $v$ and $L_1$ (note that this might not be primitive). Since any element of $\tilde{L}_1$ can be written as $av + bw$ with $a,b \in \Z$ and $w \in L_1$, we compute $$(av + bw, av + bw) = a^2(v,v) + 2ab(v,w) + b^2(w,w),$$
 and since $(w,w) \in 2\Z$, $(w,v) \in \Z$ (because $v \in L_1^\vee$) and $(v,v) \in 2\Z$ (because $v \in L_2$) we conclude that $\tilde{L}_1$ must be even, and hence integral. But since $L_1 \subset \tilde{L}_1$ and $L_1$ is maximal by assumption, we must have that $L_1 = \tilde{L}_1$, i.e., that $v \in L_1$. It follows that $L \subset L_1^\vee \cap L_2 \subset L$. Thus we prove the injectivity of  $ H_1^{\perp} / H_1 \rightarrow A/(H_1 \oplus H_2)$. But by cardinality reasons, this map must also be surjective, and hence it is an isomorphism. Carrying out the specular argument for $L_2$, we have constructed natural isomorphisms $$A_1 \cong  H_1^{\perp} / H_1 \cong A/(H_1 \oplus H_2) \cong  H_2^{\perp} / H_2 \cong A_2,$$ and point (1) of the Lemma is proved.

 \item To prove point (2), let $D = \mathrm{discr}(L_1)$. Then, since $|H_1| = |H_2|$ is coprime to $D$, we must have a natural decomposition $A= H_1 \oplus H_2 \oplus A'$, where $A'$ corresponds to the $D$-torsion of $A$. From this one easily shows that $H_1^{\perp} = H_1 \oplus A' $ and $H_2^{\perp} = H_2 \oplus A'$, and proceeds as before.
 
 \end{enumerate}
 \end{proof}
\begin{prop}
Let $X/\C$ be a K3 surface. For any isotropic line $\ell \subset \Br(X)[d]$ one can constuct a well-defined K3 surface $X_\ell$ endowed with a Hodge isometry $f_\ell$ such that $f_\ell(\mathrm{T}(X_\ell))$ is the $d$-neighbor of $\mathrm{T}(X)$ associated to $\ell$, and such that $\NS(X) \cong \NS(X_\ell)$.
\end{prop}
\begin{proof}
Denote by $\mathrm{T}_\ell(X) \subset \mathrm{T}(X)_\Q$ the $d$-neighbor associated to $\ell$. By the lemma above, there is a natural isomorphism between $A_X$ and the discriminant form of $\mathrm{T}_\ell(X)$ - which we denote by $A$. Consider now the direct sum $\mathrm{T}_\ell(X) \oplus \NS(X) $ endowed with its natural Hodge structure. We get natural identifications $A \cong - A_X \cong -\NS(X)^\vee/ \NS(X)$, which we use to glue the direct sum into the K3 lattice. By the global Torelli, we have constructed a unique projective K3 surface $X_\ell$.
\end{proof}
In order to construct the auxiliary K3 surface, we need to ensure that $\ker(\ell)$ embedds primitively into the K3 lattice.  
\begin{prop}
Let $X/\C$ be a K3 surface and assume that $\rho(X) \geq 3$. For any isotropic line $\ell \subset \Br(X)[d]$ one can construct a K3 surface $Y_\ell$ endowed with a Hodge isometry $g_\ell$ such that $g_\ell(\mathrm{T}(Y_\ell)) = \ker(\ell)$.
\end{prop}
\begin{proof}
Let $A$ be the discriminant form of $\ker(\ell)$ and let $p$ be a prime number. If $p$ divides $\mathrm{discr}(\NS(X))$ then $A \otimes \Z_p \cong A_X \otimes \Z_p$ and $A \otimes \Z_p \cong \Z/p \oplus \Z/p\Z$ if $p$ divides $d$ and $A \otimes \Z_p = 0$ otherwise. If $\rho(X) \geq 3$ we know from \cite[Corollary 1.12.3]{MR525944} that $\ker(\ell)$ embedds primitively into the K3 lattice, and hence that there is a K3 surface $Y_\ell$ with a Hodge isometry that maps $\mathrm{T}(Y_\ell)$ to $\ker(\ell)$. This is well-defined up to Fourieri-Mukai partners. 
\end{proof}
Also the next is a consequence of Nikulin results. 

\begin{lemma}
In order for $Y_\ell$ to be unique, it is enough that $\rho(X) \geq 4$ and that $\ell(A_X) \leq \rho(X) -2$. 
\end{lemma}

\begin{proof}
We only need to show that in this case $\ker(\ell)$ embedds uniquely into the K3 lattice. This follows from \cite[Corollary 1.12.3]{MR525944}. 
\end{proof}
The next follows directly from \cite[Proposition 6.6]{MR893604}.
\begin{prop} \label{prop Y}
Assume that $Y_\ell$ exists. Then, both $X$ and $X_\ell$ are moduli spaces of sheaves over some Fourier-Mukai partner of $Y_\ell$.
\end{prop}

When $Y_\ell$ does not exists, then the relation between $X$ and its neighbors can be understood in terms of twisted derived equivalences as explained by Huybrechts \cite{MR4014777}. 

\subsection{More neighbors} \label{more neighbors}
In this section we show one way to generalize Kneser correspondence to the case when $d$ and $\mathrm{discr}(L)$ are not coprime. Let $L$ be an even lattice, for any prime $p$, we denote by $L_p = L \otimes \Z_p$. We observe that the notion of neighbors can be considered also for $p$-adic lattices (even, if $p=2$) and in general over every principal ideal domain (see \cite[Definition 3.1.1]{zbMATH07417850}). To `glue' the local notion to the global one, one proceeds as follows. For any prime $p$ such that $L_p$ is unimodular (that is, for any prime $p$ coprime to $\mathrm{discr}(L)$) the $p^n$-neighbors of $L_p$ correspond to isotropic lines in $L_p \otimes \Z_p/p^n\Z_p \cong L/p^nL.$ Decompose $d = \prod_p p^{n_p}$ and let $\ell_p \subset L_p/p^{n_p}L_p$ be isotropic lines. Let $L'_p \subset L_p \otimes \Q_p$ be the neighbor of $L_p$ corresponding to $\ell_p$, where we put $L'_p = L_p$ if $\ell_p = 0$. By strong approximation, there is a unique lattice $L' \subset L_\Q$ such that $L' \otimes \Z_p = L'_p$ for every $p$. This is the $d$-neighbor associated to $\ell = \prod_p \ell_p \subset L/dL$. For an even lattice $L$ and a prime number $p$ one has a modular decomposition $$L_p \cong U_0 \perp U_1(p) \perp \cdots \perp U_n(p^n)$$ where $U_i$ is a $p$-adic unimodular lattice (even, if $p=2$) and $U_i(p^n)$ denotes the lattice $U_i$ with form scaled by a factor of $p^n$. The lattice $U_0$ is a maximal unimodular sublattice of $L_p$, and it follows that one has an isomorphism
$$A_{L_p} = A_{U_1(p)} \perp \cdots \perp A_{U_n(p^n)}.$$

Since a $p$-adic lattice (even, if $p=2$) is uniquely determined by its rank, its discriminant form, and its discriminant modulo squares, one deduces that all the maximal unimodular sublattices of a $p$-adic lattice are isometric.

For any integer $k \geq 0$, we say that a subgroup of $L_p / p^kL_p$ is \textit{non-degenerate} if it is the form $U_0/p^nU_0$ for some maximal unimodular sublattice $U_0 \subset L_p$. Similarly, a subgroup of $\Hom(L_p, \Z/p^n\Z)$ is non-degenerate if it is the image of $U_0$ under the natural map $L_p \rightarrow L_p/p^nL_p \rightarrow  \Hom(L_p, \Z/p^n\Z)$. 
\begin{defi}
\begin{itemize}
\item A line $\ell \cong \Z/p^n\Z \subset L_p/p^nL_p$ is generalized isotropic if it is isotropic with respect to the natural pairing on $L_p/p^nL_p$ and if it is contained in a non-degenerate subgroup. The analogue definition applies to a line $\ell' \subset \Hom(L_p, \Z/p^n)$. If $L$ is an even lattice and $d$ is an integer, we say that a line $\ell \subset L/dL$ is generalized isotropic if $\ell_p$ is generalized isotropic for every prime $p$, where, as before, $\ell_p$ denotes the $p$-part of $\ell = \prod_{p | d} \ell_p$. The analogous definition applies for a line in $\Hom(L, \Z/d\Z)$. 

    \item Let $L_p$ be a $p$-adic lattice (even, if $p=2$). A generalized $p^n$-neighbor is a lattice $L'_p \subset L_p \otimes \Q_p$ that can be written as $L'_p = U_0' \oplus (U_0)^{\perp}$ for some maximal unimodular sublattice $U_0 \subset L_p$, where $U_0'$ is a $p^n$-neighbor of $U_0$. From the global point of view, let $L$ be an even lattice and let $L' \subset L \otimes \Q$. Then $L'$ is a generalized $\prod_p p^{n_p} = d$-neighbor of $L$ if $L' \otimes \Z_p$ is a generalized $p^{n_p}$-neighbor of $L \otimes \Z_p$ for every $p$. 

\item A sublattice $M_p \subset L_p$ is said to be generalized $p^n$-split if it is of the form $M'_p \perp U_0^\perp$ for some maximal unimodular sublattice $U_0 \subset L_p$ and some $p^n$-split sublattice $M'_p$ of $U_0$. Similarly, one has this notion in the global case. 
\end{itemize}
\end{defi}
Clearly, one has a bijection between generalized isotropic lines in $L/dL$ and generalized isotropic lines in $\Hom(L, \Z/d\Z)$ for every $d>1$, and, if $d$ is coprime to $\mathrm{discr}(L)$, all these notions recover the original ones of Kneser. We need the following generalization of Lemma \ref{lemma transverse}:
\begin{lemma} \label{gen lemma transverse}
Let $M_p$ be a generalized $p^n$-split sublattice of a $p$-adic lattice $L_p$. Let $H_p \cong \Z/p^n\Z \subset A_{M_p}$ be the isotropic subgroup associated to the overlattice $M_p \subset L_p$. There exists a unique $H'_p \cong \Z/p^n\Z \subset A_{M_p}$ such that $H_p \cap H'_p =0$, $H'_p$ is isotropic, and one has an isomorphism of quadratic forms $(H'_p)^\perp \cong H'_p \perp  A_{L_p}$.
\end{lemma}

\begin{proof}
By definition, we can write $M_p = M'_p \perp U_0^\perp$ for some maximal unimodular sublattice $U_0 \subset L_p$ and some $p^n$-split sublattice $M_p'$ of $U_0$. We write $A_{M_p} = A_{M'_p} \perp A_{L_p}$. From Lemma \ref{lemma transverse} we know that there exists a unique $H'_p \subset A_{M'_p}$ such that $H'_p$ is isotropic, $H'_p \cap H_p = 0$ and $A_{M'_p} = H_p \oplus H'_p$. It follows that $H'_p$ satisfies the properties required in the statement. To show the uniqueness of $H'_p$, let $H''_p$ be another cyclic subgroup generated by $a +b +c \in A_{M_p} = H_p \oplus H'_p  \perp A_{L_p} $. Since $(H''_p)^\perp \cong H''_p \perp A_{L_p}$ by assumption, we see that for every $d \in A_{L_p}$ one has $(c,d) = 0$. But then $c = 0$, and from the fact that $(a,b,0)$ is isotropic one concludes that $H'_p=H''_p$. 
\end{proof}

Our generalization of Kneser neighbors is then rephrased in the following. 
\begin{prop} \label{gen knes}
Let $L$ be an even lattice and let $d$ be any integer. Then, there is a natural one-to-one correspondence between:

\begin{enumerate}
    \item Generalized $d$-neighbors of $L$;
    \item Generalized isotropic lines in $\Hom(L, \Z/d\Z)$;
    \item Generalized $d$-split sublattices of $L$.
\end{enumerate}
Moreover, if $L'$ is a generalized $d$-neighbor of $L$, there is a natural isomorphism $A_{L'} \cong A_{L}$.

\end{prop}  

\begin{proof}
It is enough to show this correspondence for $p$-adic lattices. Let $L'_p$ be a generalized $p^n$-neighbor of a $p$-adic lattice $L_p$ (even, if $p=2$). By definition, we can write $L'_p = U_0' \perp U_0^\perp$ where $U_0$ is a maximal unimodular sublattice and $U_0'$ is a classical $p^n$-neighbor of $U_0$. Let $\ell_p = \im(p^nL'_p \rightarrow L_p \rightarrow L_p/p^nL_p)$. Since we know that $\im(p^nU'_0 \rightarrow U_0 \rightarrow U_0/p^nU_0)$ is an isotropic line due to Kneser theory, we deduce that $\ell_p \subset U_0/p^nU_0 \subset L_p/p^nL_p$ is a generalized isotropic line. Similarly, let $\ell_p \subset \Hom(L_p, \Z/p^n\Z)$ be a generalized isotropic line belonging to some non-degenerate subgroup $U_0/p^nU_0$. Denote by $\ell'_p \subset \Hom(U_0, \Z/p^n\Z)$ the same isotropic line. Then $\ker(\ell_p) = \ker(\ell'_p) \perp U_0^\perp$ is a generalized $p^n$-split sublattice of $L_p$. Finally, let $M_p \subset L_p$ be a generalized $p^n$-split sublattice and let $H'_p \subset A_{M_p}$ be the unique isotropic subgroup given by the lemma above. Then, the associated overlattice $L'_p$ is easily seen to be a generalized $p^n$-neighbor. 

The natural isomorphism between $A_{L'}$ and $A_L$ also follows from Lemma \ref{gen lemma transverse}. 
\end{proof}

\begin{rmk} \label{remark generalized isotropic lines}
Let $L$ be a lattice that satisfies $\ell(A_L) \leq \mathrm{rank}(L) -3$ and let $N_d(L)$ be the number of its generalized $d$-neighbors. Under these assumption, for every prime $p$, $L/p^nL$ contains a non-degenerate subgroup of the form $(\Z/p^n\Z)^3$. It follows that $N_d(L) \geq d$ for every $d$ and every such $L$.  
\end{rmk}

We can extend this to K3 surfaces in an analogue way as we did before. Then only difference is about the existence of $Y_\ell$ for a generalized isotropic line: the condition $\rho(X) \geq 3$ must be replaced by $\ell(A_L) + 3 \leq \rho(X)$ in order to ensure that every generalized $d$-split sublattice of $\mathrm{T}(X)$ embedds primitively into the K3 lattice.

\subsection{An example} In this section we give an explicit example of the theory just exposed by proving the following proposition:
\begin{prop}
Let $X = \mathrm{Km}(E_1 \times E_2)$ be the Kummer surface associated to the product of two non-isogenous, non-CM elliptic curves $E_1$ and $E_2$. Then, for any odd integer $d$, the $d$-neighbors of $X$ are all of the form $\mathrm{Km}(E_1/C_1 \times E_2/C_2)$ where $C_1 \cong C_2 \cong \Z/ d \Z$. 
\end{prop}
We recall that the Kummer surface $X$ associated to an abelian surface $A$ is the blow up of $A/ \{ \pm 1 \}$ at its sixteen notes. We also recall that $\mathrm{T}(X)$ uniquely determines $X$ in this case, and that one always has a natural Hodge isometry between $\mathrm{T}(A)(2)$ and $\mathrm{T}(X)$, where the former denotes the transcendental lattice of $A$ with form scaled by a factor of $2$.

Let $E_1, E_2 $ be two non-isogenous, non-CM elliptic curves and let \mbox{$X= \mathrm{Km}(E_1 \times E_2)$} be the associated Kummer surface. One has $\mathrm{T}(A) \cong \mathrm{H}^1(E_1, \Z) \otimes \mathrm{H}^1(E_2, \Z)^\vee.$ Here, the dual $\mathrm{H}^1(E_2, \Z)^\vee$ plays a role mostly due to a reason of weights in Hodge theory. Let $\{e_1,f_2\}$ and $\{e_2,f_2\}$ be basis of $\mathrm{H}^1(E_1, \Z)$ and $\mathrm{H}^1(E_2, \Z)$ respectively, such that $(e_i, f_i) = 1$. Then, $v_1 = e_1 \otimes e_2$, $v_2 = e_1 \otimes f_2$, $v_3 = f_1 \otimes e_1$ and $v_4 = e_2 \otimes f_2$ form a basis of $\mathrm{T}(A)$. With respect to this basis, the Gram matrix of $\mathrm{T}(A)$ takes the following form:
 $$\begin{bmatrix}
    0 & 0 & 0 & 1 \\
    0 & 0 & -1 & 0\\
    0 & -1 & 0 & 0\\
    1 & 0 & 0 & 0
  \end{bmatrix}$$
Under the natural isomorphism $$\mathrm{T}(A)  \cong \Hom(\mathrm{H}^1(E_1, \Z), \mathrm{H}^1(E_2, \Z)),$$ 
an element $f \in \Hom(\mathrm{H}^1(E_1, \Z), \mathrm{H}^1(E_2, \Z))$ represented by the matrix $\begin{bmatrix}
    a & b \\
    c & d 
  \end{bmatrix}$ with respect to the basis $\{e_i, f_i\}$ is sent to $b v_1 + d v_2 - av_3 -c v_4$. Thus, one sees that $(f,f) = \det(f)$. Since $\mathrm{T}(A)$ is unimodular, we can work with any integer $d > 1$. From the computation above, an isotropic line $\ell \subset \mathrm{T}(A) \otimes \Z/d\Z$ is generated by an endormorphism $f \colon \mathrm{H}^1(E_1, \Z/d\Z) \rightarrow \mathrm{H}^1(E_2, \Z/d\Z)$ such that $\mathrm{ker}(f) = \mathrm{im}(f) \cong \Z/ d \Z$. In particular, we have a one-to-one correspondence 
  $$\{ d-\text{isotropic lines of } \mathrm{T}(A) \otimes \Z/ d \Z \} \leftrightarrow \{ (C_1, C_2) \colon C_i \subset E_i(\C)_{\mathrm{tors}}, C_1 \cong C_2 \cong \Z/ d\Z\}.$$
We show now how to associate the corresponding neighbor of $\mathrm{T}(A)$. Let $f \in \mathrm{T}(A) \otimes \Z/ d \Z$ be a generator of an isotropic line, and let $C_1 = \ker(f)$ and $C_2 = \mathrm{im}(f)$, when we see $f$ as a map $\mathrm{H}^1(E_1, \Z/ d\Z) \rightarrow \mathrm{H}^1(E_2, \Z/ d\Z))$. Let $E_i' \coloneqq E_i/C_i $ and put $H_i = \mathrm{im}(\mathrm{H}^1(E_i', \Z) \rightarrow \mathrm{H}^1(E_i, \Z))$, where the images are taken under the pullback in cohomology of the natural quotient map. 
Up to changing the bases, we can assume that $H_i = \Z de_i \oplus \Z f_i$ for $i= 1,2$. Consider $\mathrm{T}' \coloneqq \frac{1}{d} (H_1 \otimes H_2) \subset \mathrm{T}(A) \otimes \Q$, which is generated by $dv_1, v_2, v_3, d^{-1}v_4$. The lattice $\mathrm{T}'$ is isometric to $\mathrm{T}(A)$ and it is clearly a $d$-neighbor of $\mathrm{T}(A)$. Also, one has that $\mathrm{T}' \cap \mathrm{T}(A)$ is generated by $dv_1, v_2, v_3, v_4$, and a quick computation shows that this is the kernel of the isotropic Brauer line. 

The only thing left to show is that the K3 surface associated to $\mathrm{T}'$ is isomorphic to $X' = \mathrm{Km}(E_1' \times E_2')$, i.e., that $\mathrm{T}'$ and $\mathrm{T}(X')$ are Hodge isometric. If we endow $H_i$ with the polarization coming from the inclusion into $\mathrm{H}^1(E_i, \Z)$ we see that $H_i \cong \mathrm{H}^1(E_i', \Z)(d)$. It follows that $H_1 \otimes H_2$ is Hodge isometric to $\mathrm{H}^1(E_1', \Z)(d) \otimes \mathrm{H}^1(E_2', \Z)(d)$ and therefore that $\frac{1}{d} (H_1 \otimes H_2)$ is Hodge isometric to $\mathrm{H}^1(E_1', \Z) \otimes \mathrm{H}^1(E_2', \Z)$.

\section{Fields of definitions} \label{section field of definition}
In this section we determine a field of definition for every K3 surface in $\mathcal{H}_X$. We start with a K3 surface $X / K$ over a number field $K \subset \C$, we pick a K3 surface $(Y,f) \in \mathcal{H}^{m}_{\bar{X}}$ and we determine whether $(Y,f)$ can be defined over $K$ as well. This means that there should exist a K3 surface $Y'/K$ with a Hodge isometry $f' \colon \mathrm{T}(Y'_\C, \Q) \rightarrow \mathrm{T}(X_\C, \Q)$ such that $f'$ is defined over $K$ and such that there is an isomorphism $h \colon \overline{Y'} \xrightarrow{\sim} Y$ such that $f \circ h^*= f'$, where $h^*$ denotes the natural pullback map on the transcendental lattices. 
\begin{thm} \label{main theorem}
Let $(Y,f) \in \tilde{\mathcal{H}}_X$ and assume that $X$ is defined over a number field $K$. Denote by $\rho_X$ the Galois representation associated to $T(\overline{X}, \A_f)$. Then, the data $(Y,f, \NS(Y))$ can defined over $K$ if and only 
\begin{enumerate}
    \item The $\widehat{\Z}$-lattice $f(\mathrm{T}(Y)) \otimes \widehat{\Z}$ is stable under $\rho_X$;
    \item The induced $G_K$-representation on $A_Y$ is trivial.
\end{enumerate}
\end{thm}
Note that the analogue result for abelian varieties is almost tautological, and the only fundamental difference between the two situations is point (2).

\begin{proof}
Let $X/ \overline{\Q}$ be a smooth projective variety and let $\sigma \in \Aut(\overline{\Q})$. Denote by $X^\sigma$ the fibre product of $X$ along $\sigma$, and denote by $\sigma_X \colon X \rightarrow X^\sigma$ the induced isomorphism of schemes. Note that $ X^{\sigma}$ is the scheme obtained by applying $\sigma$ to the defying equation of $X$. When $X$ is a K3 surface, we denote by $\sigma_X^* \colon \mathrm{H}^2_{\et}(X^\sigma, \widehat{\Z}(1)) \rightarrow \mathrm{H}^2_{\et}(X, \widehat{\Z}(1))$ the induced isomorphism in \et ale cohomology. Moreover, we write $\sigma_{X, \mathrm{alg}}^*$ and $\sigma_{X, \mathrm{tr}}^*$ to mean respectively the restriction of $\sigma_X^*$ to the algebraic and transcendental part of $\mathrm{H}^2_{\et}(X^{\sigma}, \widehat{\Z}(1))$. 
To produce a model of $X$ over $K$ we use Galois descent and the global Torelli theorem for K3 surfaces. We refer the reader to \cite{poonen2017rational} for a general introduction to Galois descent, and also to \cite{valloni2019fields}, where a similar argument is used. Essentially, we will produce a Galois descent data for $X$ over $K$, which consists, for every $\sigma \in G_K$, of an isomorphism $f_{\sigma} \colon X^{\sigma} \rightarrow X$ such that the assignment $\sigma \mapsto f_\sigma$ satisfies the cocycle condition $f_{\sigma \tau} = f_\sigma \circ (f_\tau)^\sigma$. Corollary 4.4.6 together with Remark 4.4.8 of \cite{poonen2017rational} says that if such data exists, then $X$ admits a model $X_0/K$. This model is characterized by the fact that the Galois action on its $\overline{\Q}$-points is given by the composition 
$$
    X(\overline{\Q}) \xrightarrow{\sigma} X^\sigma(\overline{\Q}) \xrightarrow{f_\sigma} X(\overline{\Q}),
$$
for each $\sigma \in G_K$. Similarly, one has a natural identification $\overline{X}_0 \cong X$ and then the Galois action on $\mathrm{H}^2_{\et}(\overline{X}_0, \widehat{\Z})$ is given by the composition \begin{equation} \label{Equation Galois action} \mathrm{H}^2_{\et}(X, \widehat{\Z}(1)) \xrightarrow{f_\sigma^* } \mathrm{H}^2_{\et}(X^\sigma, \widehat{\Z}(1)) \xrightarrow{\sigma_X^*} \mathrm{H}^2_{\et}(X, \widehat{\Z}(1)). \end{equation} 

Now, if $(Y,f,\NS(Y))$ is defined over $K$, then (1) and (2) are satisfied. To prove the other direction we argue as follows. For $\sigma \in G_K$, we write the diagram \ref{square galois}, taking into account that $\mathrm{T}(X, \widehat{\Z}) = \mathrm{T}(X^\sigma, \widehat{\Z})$ because $X$ is defined over $K$:
    \begin{center}
 \begin{tikzcd} \label{square galois II}
\mathrm{T}(Y^\sigma,\A_f) \arrow{r}{f^\sigma} \arrow{d}{\sigma_{Y, \mathrm{tr}}^*}
& \mathrm{T}(X,\A_f) \arrow{d}{\sigma_{X, \mathrm{tr}}^*} \\
\mathrm{T}(Y,\A_f) \arrow{r}{f}
& \mathrm{T}(X,\A_f).
\end{tikzcd}
    \end{center}
In particular, we have that \begin{equation} \label{equation proof 1}
    \sigma_{X, \mathrm{tr}}^* \circ f^\sigma = f \circ \sigma_{Y, \mathrm{tr}}^*.
\end{equation}
    Due to the fact that $f(\mathrm{T}(Y,\widehat{\Z}))$ is $G_K$-stable, and that $\sigma_{Y, \mathrm{tr}}^*$ induces an identification between $\mathrm{T}(Y^\sigma,\widehat{\Z})$ and $\mathrm{T}(Y,\widehat{\Z})$, we deduce from the diagram above that 
    
    \begin{equation} \label{equation 1 proof field of defi}
        f^{\sigma}(\mathrm{T}(Y^\sigma,\widehat{\Z})) = f(\mathrm{T}(Y,\widehat{\Z})).
    \end{equation} Hence $$f^{\sigma}(\mathrm{T}(Y^\sigma)) = f^{\sigma}(\mathrm{T}(Y^\sigma,\widehat{\Z})) \cap \mathrm{T}(X,\Q) = f(\mathrm{T}(Y,\widehat{\Z})) \cap  \mathrm{T}(X,\Q)  = f(\mathrm{T}(Y)).$$ In particular, the map $(f^{\sigma})^{-1} \circ f \colon \mathrm{T}(Y, \Q) \rightarrow \mathrm{T}(Y^\sigma, \Q)$ sends $\mathrm{T}(Y)$ to $\mathrm{T}(Y^\sigma)$, i.e., it is a integral Hodge isometry. Our aim is to show that the isometry 
    \begin{equation} \label{equation extending}
        ((f^{\sigma})^{-1} \circ f) \oplus ( \sigma_{Y, \mathrm{alg}}^*)^{-1}  \colon \mathrm{T}(Y) \oplus \NS(Y) \rightarrow \mathrm{T}(Y^{\sigma}) \oplus \NS(Y^{\sigma})
    \end{equation} 
    extends to a global Hodge isometry $\mathrm{H}^2_{\text{sing}}(Y, \Z)\rightarrow \mathrm{H}^2_{\text{sing}}(Y^\sigma, \Z).$ To this purpose, denote by $\tilde{\sigma}_{Y^\sigma, \mathrm{tr}}^{*} \colon \mathrm{T}(Y, \widehat{\Z}) \rightarrow \mathrm{T}(Y, \widehat{\Z})$ the unique map that satisfies 
\begin{equation} \label{equation ext proof 1}
    f^\sigma \circ \tilde{\sigma}_{Y^\sigma, \mathrm{tr}}^{*} = \sigma_{X, \mathrm{tr}}^* \circ f^\sigma;
\end{equation} this is nothing but the restriction of $\sigma_{X, \mathrm{tr}}^*$ to $f^\sigma(\mathrm{T}(Y^\sigma, \widehat{\Z}))$ seen as an isometry of $\mathrm{T}(Y^\sigma, \widehat{\Z})$, and it exists because of \eqref{equation 1 proof field of defi} and the fact that $f(\mathrm{T}(Y, \widehat{\Z}))$ is $G_K$-invariant. By point (2) $\tilde{\sigma}_{Y^\sigma, \mathrm{tr}}^{*}$ acts as the identity on $A_{Y^\sigma} \cong A_Y$. From equation \eqref{equation proof 1} and \eqref{equation ext proof 1} we have that $ f^\sigma \circ \tilde{\sigma}_{Y^\sigma, \mathrm{tr}}^{*} = f \circ \sigma_{Y, \mathrm{tr}}^*$, hence that $(f^{\sigma})^{-1} \circ f = \tilde{\sigma}_{Y^\sigma, \mathrm{tr}}^{*} \circ  (\sigma_{Y, \mathrm{tr}}^*)^{-1}.$ From this we conclude that the maps $A_{Y} \rightarrow A_{Y^\sigma}$ induced respectively by $(f^{\sigma})^{-1} \circ f $ and $ (\sigma_{Y, \mathrm{tr}}^*)^{-1}$ coincide. But also the maps $A_{Y} \rightarrow A_{Y^\sigma}$ induced by $(\sigma_{Y, \mathrm{tr}}^*)^{-1}$ and $(\sigma_{Y, \mathrm{alg}}^*)^{-1}$ coincide, because both are restrictions of global $\widehat{\Z}$-isometry $(\sigma_Y^*)^{-1}$. Using Proposition \ref{Nikulin extension}, we conclude that there exists a global Hodge isometry $\mathrm{H}^2_{\text{sing}}(Y^\sigma,\Z) \rightarrow \mathrm{H}^2_{\text{sing}}(Y,\Z) $ which extends \eqref{equation extending}, and since $( \sigma_{Y, \mathrm{alg}}^*)^{-1} $ sends ample line bundles to ample line bundles, this is in turn induced by a unique isomorphism $g_\sigma \colon Y \xrightarrow{\sim} Y^\sigma$ by the global Torelli theorem. Then, we are left to show that the association $\sigma \mapsto (g_\sigma \colon Y \rightarrow Y^\sigma)$ satisfies the cocycle condition, i.e., that $g_\sigma \circ (g_\tau)^{\sigma} = g_{\sigma \tau}$ for any $\sigma, \tau \in G_K$. It is sufficient to check this identity cohomologically. At the level of N\'{e}ron-Severi lattices, we have that $g_\tau^* = (\tau_{Y, \mathrm{alg}}^*)^{-1}$, hence that $((g_\tau)^{\sigma})^* = (\sigma_{Y^{\tau},\mathrm{alg}}^*)^{-1} \circ g_\tau^*  \circ \sigma_Y^* = (\sigma_{Y^{\tau},\mathrm{alg}}^*)^{-1} \circ (\tau_{Y, \mathrm{alg}}^*)^{-1} \circ \sigma_{Y, \mathrm{alg}}^*$. Since $(\sigma \tau)_Y = \sigma_{Y^\tau} \circ \tau_Y$, we readily conclude in this case. At the level of transcendental lattices, we have that $g_\tau^* = (f^\tau)^{-1} \circ f$ and hence that $ ((g_{\tau})^{\sigma})^* = (\sigma_{Y^{\tau},\mathrm{tr}}^*)^{-1} \circ [(f^\tau)^{-1} \circ f] \circ \sigma_{Y, \mathrm{tr}}^*$. Note that $f^\sigma = (\sigma_{X, \mathrm{tr}}^*)^{-1} \circ f \circ \sigma_{Y, \mathrm{tr}}^*$ from which we see that $ ((g_{\tau})^{\sigma})^* \circ g_\sigma^* = (\sigma_{Y^{\tau},\mathrm{tr}}^*)^{-1} \circ (f^\tau)^{-1} \circ \sigma_{X, \mathrm{tr}}^* \circ f = (f^{\sigma \tau})^{-1} \circ f$. Thus, the cocycle condition is proved, and we check finally that both $f$ and $\NS(\overline{Y})$ are defined over $K$. Consider the model $Y_0/K$ constructed with such Galois descent data. We need to show that $G_K$ acts trivially on $\NS(\overline{Y}_0)$. Using \eqref{Equation Galois action}, we know that $\sigma \in G_K$ acts on $\mathrm{H}^2_{\et}(\overline{Y}_0, \widehat{\Z}(1)) \cong \mathrm{H}^2_{\et}(Y, \widehat{\Z}(1))$ as the composition $\sigma_{Y}^* \circ g_\sigma^* $. It is clear that this is the identity on the N\'{e}ron-Severi group, and we are only left to check that $f \circ ( \sigma_{Y, \mathrm{tr}}^* \circ (g_\sigma^*)_\mathrm{tr}) = \sigma_{X, \mathrm{tr}}^* \circ f$ and since $(g_\sigma^*)_\mathrm{tr} = (f^{\sigma})^{-1} \circ f$ this becomes equation \eqref{equation proof 1}.
\end{proof}
\begin{rmk}
Note that the model of $Y$ thus constructed depends on the chosen model of $X$ over $K$.
\end{rmk}
Let $K'/K$ be the smallest extension that satisfies (1) and (2) in the theorem above. Then, this is also the smallest extension over which $(Y,f, \NS(Y))$ can be defined. Hence, the data $(Y,f, \NS(Y))$ admits a smallest field of definition over $K$. 

A direct consequence of Proposition \ref{prop def hodge isom} and Lemma \ref{lemma Neron-Severi} is that a minimal field of definition for $(Y,f, \NS(Y))$ is not too far away from a field of definition for $Y$ of smallest degree:

\begin{cor} \label{corollary minimal field}
Let $L$ be a field of definition for $Y$, and let $K'/K$ be the minimal field extension of $K$ that satisfies (1) and (2) of the theorem above. Then, there exists an extension $L' /L$ of degree $[L' \colon L] \leq C_n \cdot C_h$ such that $K' \subset L'$.  
\end{cor}
Another corollary comes from Lemma \ref{Lemma on iso of discriminant forms}:
\begin{cor} \label{Proposition easier for genus}
Let $(Y,f) \in \tilde{\mathcal{H}}_X$, and assume that $(X, \NS(X))$ can be defined over a number field $K$. Assume that $f(\mathrm{T}(Y))$ and $\mathrm{T}(X)$ belong to the same genus, then:
    \begin{enumerate}
        \item If $X$ (and hence also $Y$) is maximal, $(Y,f, \NS(Y))$ is defined over $K$ if and only if $f(\mathrm{T}(Y)) \otimes \widehat{\Z}$ is $G_K$-stable. 
        \item If $f$ is such that $[\mathrm{T}(X) \colon f(\mathrm{T}(Y)) \cap \mathrm{T}(X)]$ is coprime to $\mathrm{discr}(\mathrm{T}(X))$, then the same conclusion of (a) is true.  
    \end{enumerate}
\end{cor}

\begin{proof}
All the cases follow in the same way, as an application of Lemma \ref{Lemma on iso of discriminant forms} and Theorem \ref{main theorem}. To prove case (1), for example, one notes that if $G_K$ stabilizes $f(\mathrm{T}(Y)) \otimes \widehat{\Z} =: \mathrm{T}$, the the naturality in Lemma \ref{Lemma on iso of discriminant forms} implies that the isomorphism between $A_X = \mathrm{T}(X)^\vee / \mathrm{T}(X)$ and $A_{\mathrm{T}} = \mathrm{T}^\vee / \mathrm{T}$ is $G_K$-equivariant, and since $G_K$ acts trivially on $A_X$ by assumption, the statement follows from Theorem \ref{main theorem}.  
\end{proof}
Note in particular that maximal K3 surfaces behave, in this respect, like abelian varieties. Concerning generalized neighboring K3 surfaces, we have the following 
\begin{cor} \label{proposition field of definition nhbrs}
Let $X/K$ be a K3 surface over a number field $K$ and let $\ell \subset \Br(\overline{X})[d]$ be a generalized isotropic line. Assume that $G_K$ acts trivially on $\NS(\overline{X})$.  
\begin{enumerate}
    \item If $\ell$ is $G_K$-stable, then there is a K3 surface $X_\ell / K$ with $\NS(\overline{X}_{\ell}) = \NS(\overline{X}_{\ell})^{G_K}$ and a Hodge isometry $f \colon \mathrm{T}(X_{\ell,\C}, \Q) \rightarrow \mathrm{T}(X, \Q)$ defined over $K$ such that $f(\mathrm{T}(X_{\ell, \C}))$ is the $d$-neighbor of $\mathrm{T}(X_\C)$ associated to $\ell$.
    \item Assume also that $G_K$ acts trivially on $\ell$, and that $\rho_{\bar{X}} \geq 3$. Then $(Y_\ell, g_\ell, \NS(Y_\ell))$ can be defined over $K$. 
\end{enumerate}
\end{cor}
\begin{proof}
Point (1) follows directly from the last part of Proposition \ref{gen knes}. 
To prove point (2), we need to prove that if $G_K$ acts trivially on $\ell$ then it also acts trivially on the discriminant form of $\ker(\ell)$. This follows from Lemma \ref{gen lemma transverse}: let $A$ be the discriminant form of $\ker(\ell)$ and let $H \subset A$ be the isotropic subgroup associated to $\ker(\ell) \subset \mathrm{T}(X)$ and let $H' \subset A$ its unique transverse as in \ref{gen lemma transverse}. It follows that $H' \cong \Hom(H, \Q/ \Z) \cong \ell$, and therefore, if $G_K$ acts trivially on $\ell$ and on $\NS(\overline{X})$ then it also acts trivially on $A$.  
\end{proof}
Our last corollary of the theorem concerns moduli spaces of sheaves on K3 surfaces. Let $Y$ be a complex K3 surface, let $v \in \mathrm{H}^*(Y, \Z)$ be a primitive and isotropic Mukai vector, and let $X=M(v)$ be the associated moduli space. Assume that $X$ is a K3 surface. Let $f \colon \mathrm{T}(Y, \Q) \rightarrow \mathrm{T}(X, \Q)$ be the isogeny induced by the quasi-universal family, so that $f(\mathrm{T}(Y)) \subset \mathrm{T}(X)$. 
\begin{cor} \label{corollary moduli space}
Assume that $X$ has a model $X'$ over a number field $K \subset \C$ and assume that $G_K$ acts trivially on $\NS(\overline{X'})$ and on $A_Y$ as in point (2) of \ref{main thm introduction}. Then, there exists a model $Y'/K$ of $Y$ such that $X'$ is the moduli space of stable sheaves on $Y$. 
\end{cor}
\begin{proof}
We know by \ref{main thm introduction} that under these conditions there exists a K3 surface $Y'/K$ together with a $K$-Hodge isometry $f' \colon \mathrm{T}(Y'_\C, \Q) \rightarrow \mathrm{T}(X'_\C, \Q)$ and an isomorphism $h \colon Y'_\C \rightarrow Y$ such that $f' \circ h^*_\mathrm{tr} = f$. 

Let $h^* \colon \mathrm{H}^*(Y, \Z) \rightarrow \mathrm{H}^*(Y'_\C, \Z)$ be the induced map on cohomology, and let $v'=h^*(v)$. Then, $v'$ is defined over $K$ because $G_K$ acts trivially on $\NS(\overline{Y'})$. It follows that we can consider the moduli space $M(v')$ and we have a canonical isomorphism $g \colon M(v')_\C \xrightarrow{\sim} X$. We are only left to show that $M(v') \cong X'$ already over $K$, which follows from the fact that $g^* \colon \mathrm{H}_{\et}^2(\overline{X'}, \widehat{\Z}) \rightarrow \mathrm{H}_{\et}^2(\overline{Y'}, \widehat{\Z})$ commutes with the Galois action. 

\end{proof}

\section{Finiteness and uniformity} \label{finiteness and uniformity}
We begin by proving Theorem \ref{Thm implications}, which is straightforward at this point. 

\begin{proof}[Proof of Theorem \ref{Thm implications}]
 For $i \in \N$, let $X_i$ be K3 surfaces defined over number fields $K_i$ with $[K_i \colon \Q] \leq N$ and let $d_i>0$ be a sequence of integers such that $\Br(\overline{X}_i)[d_i] \subset \Br(\overline{X}_i)^{G_K}$. We can assume without loss of generality that $\NS(\overline{X}_i)^{G_{K_i}} = \NS(\overline{X}_i)$. Assume that every $X_i$ satisfies $\ell(A_X) \leq 19 - \rho(X)$. 
\begin{enumerate}
    \item Assume that $3 + \ell(A) \leq \rho_{\bar{X}_i}$ for every $i$. Since Shafarevich conjecture holds by assumption, let $\mathcal{L}_N$ be the finite set of lattices that can appear as $\NS(\overline{X})$ for $X/K$ a K3 surface with $[K \colon \Q] \leq N$. Let 
     $$M = \max \{ d \in \N \colon \,\, \text{there is} \,\, L_1 , L_2 \in \mathcal{L}_N \,\, \text{such that} \,\, d^2 \mathrm{discr}(L_1) = \mathrm{discr}(L_2)\},$$ 
    then $M < \infty$. Due to Remark \ref{remark generalized isotropic lines} we conclude that there are generalized $d_i$-isotropic lines in $\Br(\overline{X}_i)[d_i]$ for every $i>0$. Choose one of them, $\ell_i  \subset \Br(\overline{X})[d_i]$. Since $G_{K_i}$ acts trivially on $\ell_i \subset \Br(\overline{X})[d_i]$ we have a K3 surface $Y_i/ K_i$ together with a Hodge isometry that identifies $\mathrm{T}(Y_{i, \C})$ with $\ker(\ell)$. But since $\mathrm{discr}(\NS(\overline{Y}_i)) = d_i^2 \cdot \mathrm{discr}(\NS(\overline{X}_i))$ one deduces that $d_i \leq M < \infty$.
    
    \item The proof uses the same ideas, but only part (1) of Proposition \ref{proposition field of definition nhbrs}. Let $N_i$ be the number of generalized $d_i$-isotropic lines in $\mathrm{T}(X_{i, \C})$. Due again to Remark \ref{remark generalized isotropic lines} we know that $N_i \geq d_i$ for every $i$. Since $E_{X_{i \C}}$ is totally real, we conclude that there are at least $N_i$ geometric isomorphism classes of K3 surfaces that are isogenous to $X_i$ and that can be defined over $K_i$. Thus, if Conjecture \ref{conj finiteness isogenous} is true, we conclude that the $d_i$ are again bounded. 
\end{enumerate}
 If $\rho(X_i) < 3 + \ell(A_{X_i})$ then one can use the classical form of Kneser correspondence to conclude analogously for $d_i$ coprime to $\mathrm{discr}(\NS(X_i))$. In this case, we only need $3 \leq \rho_{\bar{X}_i}$ to conclude in the first point. 
\end{proof}

In this last part of the paper, we prove the finiteness stated at the end of the introduction. We use some ideas from the previous sections, but the finiteness itself is independent from \ref{main theorem}. Our main ingredient is Theorem \ref{Cadoret-Moonen proposition} of Cadoret and Moonen; we remark that the results in \cite{cadoret_moonen} builds on the Mumford-Tate conjecture for K3 surfaces, whose proof uses the Tate-conjecture for divisors of abelian varieties in a crucial way (see \cite{zbMATH01181394} and \cite{zbMATH00931955}). Thus, our finiteness is not independent from Faltings' celebrated work \cite{zbMATH03944027}.

We shall use the following notation: let $V$ be a rational vector space, and let $G \subset \Gl(V)$ be a linear algebraic group. For any maximal rank free $\Z$ module $L \subset V$ we denote by $\mathbf{G}_L$ the integral form of $G$ induced by $L$. This is constructed as the Zariski closure of $G$ seen inside the generic fibre of $\Gl(L)$. It follows that $\mathbf{G}_L$ is a flat affine group scheme over $\Spec(\Z)$ whose generic fibre recovers $G$. In case $L = \mathrm{T}(X)$ we simply write $\mathbf{G}_X$ for $\mathbf{G}_L$.   

We begin with a lemma:
\begin{lemma} \label{useful variant}
Let $Y \in \mathcal{H}^{\mathrm{low}}_X$ and assume that $X$ can be defined over a field $K$ and that $Y$ can defined over a field $K'$. Choose any Hodge isometry $f \colon \mathrm{T}(Y, \Q) \xrightarrow{\sim} \mathrm{T}(X, \Q)$ such that $f(\mathrm{T}(Y)) \subset \mathrm{T}(X)$. Then, there exists an extension $K'' / K K'$ of degree $[K'' \colon K' K] \leq C_n \cdot C_h$ such that $f(\mathrm{T}(Y))^\vee / \mathrm{T}(X)  \subset \Br(\overline{X})^{G_{K''}}.$
\end{lemma}

\begin{proof}
In fact, let $K''/ K' K$ be an extension of degree $[K'' \colon K'K] \leq C_h \cdot C_n$ over which both $\NS(\overline{Y})$ and $f$ are defined. It follows that $f(\mathrm{T}(Y))^\vee / \mathrm{T}(X) \subset \Br(\overline{X})^{G_K}$.
\end{proof}
 The next is a direct consequence of Theorem \ref{integral mt conj}, see also Corollary 5.3. of \cite{MR3830546}.
\begin{prop} \label{Cadoret-Moonen proposition}
Let $X$ be a K3 surface defined over a number field $K$ and let $N>0$ be an integer. There exists a constant $C = C(N,X)$ such that for any field extension $K'/K$ of degree bounded by $N$, one has $$|\Br(\bar{X})^{G_{K'}}| \leq C. $$
\end{prop}

\begin{proof}
Let $\mathbf{MT}_X$ be the integral form of $\MT_X$ induced by the lattice $\mathrm{T}(X) \subset \mathrm{T}(X,\Q)$. Hence,  $\mathbf{MT}_X(\widehat{\Z})$ naturally acts on $\mathrm{T}(X)^\vee \otimes \Q / \Z.$ If $\Gamma \subset \mathbf{MT}_X(\widehat{\Z})$ is a finite-index group, then the fixed points $(\mathrm{T}(X)^\vee \otimes \Q / \Z)^\Gamma$ are finitely many. Moreover, it is well-known that the group $\mathbf{MT}_X(\widehat{\Z})$ has only finitely many subgroups of a certain index, and for each $M>0$ we define $$C(X,M) \coloneqq \max \{|(\mathrm{T}(X)^\vee \otimes \Q / \Z)^\Gamma| \colon \Gamma \,\, \text{has index} \,\, \leq M \}.$$ Let now $\rho_X \colon G_K \rightarrow \Gl(\mathrm{T}(X,\Q))(\A_f) \cong \Gl_X(\A_f)$ be the Galois representation associated to $T(\overline{X},\A_f)$. Then, Theorem \ref{integral mt conj} says that the index $$M \coloneqq [\mathbf{MT}_X(\widehat{\Z}) \colon \rho_X^\vee(G_{K'}) \cap \mathbf{MT}_X(\widehat{\Z}) ]$$ is finite. For any extension $K'/K$ of degree $[K' \colon K] \leq N$ we then have
$$|\Br(\bar{X})^{G_{K'}}| = | (\mathrm{T}(\overline{X})^\vee \otimes \Q / \Z)^{\rho_X^\vee(G_{K'})}| \leq | (\mathrm{T}(\overline{X})^\vee \otimes \Q / \Z)^{\rho_X^\vee(G_{K'}) \cap \mathbf{MT}_X(\widehat{\Z}) }|,$$
and this last quantity is always smaller than $C(X,NM)$.
\end{proof}
 We recall that $\mathcal{H}^{\mathrm{low}}_X$ was defined in \ref{subsection constructing from Brauer}. 
\begin{prop} \label{Finiteness I}
Let $N>0$ be an integer and let $X/ \overline{\Q}$ be a K3 surface. Then, only finitely many elements of $\mathcal{H}^{\mathrm{low}}_X$ can be defined over some number field $K$ of degree $[K \colon \Q] \leq N.$ 
\end{prop}
\begin{proof}
Assume that $X$ is defined over a number field $K$ and suppose that there were infinitely many, pairwise distinct K3 surfaces $X_i \in \mathcal{H}^{\mathrm{low}}_X$, for $i \in \N$, that can be defined over field extensions $K_i/ K$ of degree $[K_i \colon K] \leq N$. 
The index of $\mathrm{T}(X_i)$ in $\mathrm{T}(X)$ does not depend on a chosen rational Hodge isometry. In fact, for any inclusion of lattices $N \subset L$ of finite index, one has $\mathrm{discr}(N) = \mathrm{discr}(N) \cdot [L \colon N]^2$. From the fact that $\mathrm{T}(X)$ has only finitely many subgroups of a given index, and that a K3 surface has only finitely many Fourier-Mukai partners, we conclude that the index of $\mathrm{T}(X_i)$ in $\mathrm{T}(X)$ must be unbounded for $i \in \N$. 
For any $i$, denote by $\tilde{X}_i$ a chosen model over $K_i$ and also choose a Hodge isometry $f_i \colon \mathrm{T}(X_{i}) \rightarrow \mathrm{T}(X)$ such that $f_i(\mathrm{T}(X_{i})) \subset  \mathrm{T}(X)$, and let $K_{i}'/ K_{i}$ be an extension of universally bounded degree over which both $\NS(\overline{X}_i)$ and $f_i$ are defined. Write $A_i \subset \Br(X)$ for the group $f_{i}(\mathrm{T}(X_{i}))^\vee / \mathrm{T}(X)^\vee$. From Lemma \ref{useful variant} we see that $$|\Br(\overline{X})^{G_{K_i'}}| \geq |A_i| = [\mathrm{T}(X) \colon f_i(\mathrm{T}(X_{i}))].$$ But since the indices $[\mathrm{T}(X) \colon f_i(\mathrm{T}(X_{i}))]$ are unbounded for $i \in \N$, we find a contradiction to Proposition \ref{Cadoret-Moonen proposition}.
\end{proof}

\begin{prop}
The two following assertions are equivalent. 
\begin{enumerate}
    \item Let $X/\overbar{\Q}$. Then, only finitely many elements of $\mathcal{H}_X$ can be defined over a field extension $K/\Q$ of a given degree.
    \item Let $X/\overbar{\Q}$. Then, only finitely many elements of $\mathcal{H}^{\mathrm{max}}_X$ can be defined over a field extension $K/\Q$ of a given degree.
\end{enumerate}
\end{prop}

\begin{proof}
Clearly (1)$\implies$(2). To show the reverse implication one argues as follows. Suppose that (2) is true but (1) is not. Then, there are K3 surfaces $X_i / K_i$ defined over fields $K_i$ such that $[K_i \colon \Q]$ is bounded and $\overline{X}_i \in \mathcal{H}_X$ are pairwise disjoint. We can assume without loss of generality that $G_{K_i}$ acts trivially on $\NS(\overline{X}_i)$. Let $\mathrm{T}(X_{i, \C}) \subset \mathrm{T}_i$ be a maximal lattice containing $\mathrm{T}(X_{i, \C})$. Since $G_{K_i}$ acts trivially on $A_{X_{i}}$ one shows that $G_{K_{i}}$ stabilizes $\mathrm{T}_i \otimes \widehat{\Z}$ and that moreover it acts trivially on $A_{\mathrm{T}_i}$. It follows that there are K3 surfaces $(Y_i, f_i, \NS(Y_i))$ defined over $K_i$ with a Hodge isometry $f_i \colon \mathrm{T}(Y_{i, \C}, \Q) \rightarrow \mathrm{T}(X_{i, \C}, \Q)$ such that $f_i(\mathrm{T}(Y_{i, \C})) = \mathrm{T}_i$. In particular, these are maximal K3 surfaces, and by $(2)$, they can be only finitely many geometrically. The results now follow from \ref{Finiteness I}. 
\end{proof}
Up until now, our results work for every K3 surface. In the next theorem we specialize to the generic case when $\End_{\Hdg}(\mathrm{T}(X)) = \Z$. 
\begin{thm} \label{finiteness maximal}
Let $N>0$ be an integer and assume that  $\End_{\Hdg}(\mathrm{T}(X)) = \Z$. Then, only finitely many elements of $\mathcal{H}_{X}^{\mathrm{max}}$ can be defined over a field extension $K/\Q$ of degree $[K \colon \Q] \leq N$. 
\end{thm}

\begin{proof}
We can assume without loss of generality that $X$ is maximal.
As a consequence of Shimura characterization of maximal lattices there is an isomorphism $$\mathrm{SO}_X(\A_f) / \mathbf{SO}_X(\widehat{\Z}) \xrightarrow{\sim} \mathcal{L}^{\text{max}}.$$ 
Assume that $X$ can be defined over a number field $M/\Q$, and assume without loss of generality that $\rho_X(G_M) \subset \mathrm{SO}_X(\A_f)$ (this can always be achieved after a finite extension). Then by \ref{Cadoret-Moonen proposition} $\rho_X(G_M)$ is a compact-open subgroup. Let $C_h>0$ be the constant from Proposition \ref{prop def hodge isom}, and let $\mathcal{V} \subset \rho_X(G_M)$ be the intersection of all the subgroups of $\rho_X(G_M)$ of index smaller than $C_h \cdot N$. Since there are only finitely many such subgroups, also $\mathcal{V}$ is a compact-open subgroup of $ \mathrm{SO}_X(\A_f)$. 
Let now $(Y,f) \in \tilde{\mathcal{H}}^\mathrm{max}$, so that $f(\mathrm{T}(Y)) \in \mathcal{L}^\mathrm{max}$, and assume that $Y$ can be defined over an extension $K/\Q$ of degree $[K \colon \Q] \leq N$. Then, there exists an extension $K'/KM$ of degree $[K' \colon KM] \leq C_h$ over which $f$ is defined too. Since $f$ is defined over $K'$, it follows that $f(T(\overline{Y}, \widehat{\Z}))$ is $G_{K'}$-stable, hence that $\mathcal{V} \subset \mathbf{SO}_Y(\widehat{\Z})$. Thus, to prove the statement, it suffices to show that only finitely many elements $g \in \mathrm{SO}_X(\A_f) / \mathbf{SO}_X(\widehat{\Z})$ satisfy 
\begin{equation} \label{equation that fails for E}
    \mathcal{V} \subset \mathrm{SO}_{g\mathrm{T}(X)}(\widehat{\Z}) = g \mathrm{SO}_X(\widehat{\Z}) g^{-1}.
\end{equation} 
A compact-open subgroup $\mathcal{U} \subset \mathrm{SO}_X(\A_f)$ is said to be maximal if it is not properly contained in another compact-open subgroup. A similar definition applies to compact open subgroups $\mathcal{U}_p \subset \mathrm{SO}_X(\Q_p)$. 

\begin{lemma}
If $\mathcal{U} \subset \mathrm{SO}_X(\A_f)$ is any compact open subgroup, and $\mathcal{U}_p \subset \mathrm{SO}_X(\Q_p)$ denotes its $p$-component, then $\mathcal{U}_p$ is maximal for almost every prime $p$.
\end{lemma}
\begin{proof}
Consider the natural embedding $\mathbf{SO}_X \subset \mathbf{GL}_X$. By the definition of the adelic topology, for almost every $p$ we have that $\mathcal{U}_p = \mathbf{GL}_X(\Z_p) \cap \mathbf{SO}_X(\Q_p)$. If moreover $p$ is a prime of good reduction for $\mathbf{SO}_X$, which happens for example whenever $p$ is odd and does not divide $\mathrm{discr}(\mathrm{T}(X))$, then it follows from Bruhat-Tits theory that $\mathbf{GL}_X(\Z_p) \cap \mathbf{SO}_X(\Q_p)$ is maximal (for an elementary proof of this fact one can consult \cite{MR3612325}).  
\end{proof}
Since $\mathrm{SO}_X$ is reductive, \cite[Proposition 3.16]{zbMATH00052370} and the lemma above say that every compact-open subgroup of $ \mathrm{SO}_X(\A_f)$ is contained in at least one maximal compact-open subgroup, and in at most finitely many such subgroups. 
 Assume by contradiction that there are infinitely many $g_i \in \mathrm{SO}_X(\A_f) / \mathbf{SO}_X(\widehat{\Z})$, for $i \in \N$, such that $\mathcal{V} \subset g_i \mathrm{SO}_X(\widehat{\Z}) g_{i}^{-1}$. Since $\mathcal{V}$ is contained in only finitely many maximal compact-open subgroups, and since every compact-open subgroup of $ \mathrm{SO}_X(\widehat{\Z})$ has only finitely many subgroups of a given index, we can assume without loss of generality that $g_i\mathrm{SO}_X(\widehat{\Z})g_i^{-1} = g_j \mathrm{SO}_X(\widehat{\Z}) g_j^{-1}$ for $i,j \in \mathbf{N}$. Let $\mathcal{N}$ be the normalizer of $\mathrm{SO}_X(\widehat{\Z})$ in $\mathrm{SO}_X(\A_f)$. Then we must have that $g_1 g_i^{-1}$ belongs to the image of $\mathcal{N}$ in $\mathrm{SO}_X(\A_f)/ \mathrm{SO}_X(\widehat{\Z})$ for any $i \in \N$. But since $\mathrm{SO}_X$ is semisimple, from the last part of \cite[Proposition 3.16]{zbMATH00052370} one has that the natural inclusion
    \begin{equation} \label{normalizer finite-index}
        \mathrm{SO}_X(\widehat{\Z}) \subset \mathcal{N}
    \end{equation} is of finite-index. Thus, we find a contradiction to the infinitude of the $g_i$'s.

\end{proof}
 
\begin{rmk} \label{remark problem}
When $E_X \neq \Q$ we do not have anymore the equality \eqref{equation that fails for E} for any $g \in \mathrm{SO}_X(\A_f)$, but only for $g \in \MT_X(\A_f) \subset \mathrm{SO}_X(\A_f)$. Using the same ideas of above and the description of Mumford-Tate groups of K3 surfaces of Zarhin \cite{MR697317}, we can prove the finiteness of the smaller $\mathrm{H}^{\MT_X}_X \subset \mathrm{H}_X$ given by all the K3 surfaces $Y$ with a Hodge isometry $f \colon \mathrm{T}(Y, \Q) \rightarrow  \mathrm{T}(X, \Q)$ such that $f(\mathrm{T}(Y))$ and $\mathrm{T}(X)$ belong to the same $\MT_X$-genus. 
\end{rmk}

\bibliography{Bibliografia}
\bibliographystyle{alpha}

\end{document}